\documentclass[times]{oupau-ppn}
\usepackage[colorlinks=true]{hyperref}
\hypersetup{urlcolor=blue, linkcolor=blue, citecolor=red, anchorcolor=blue}
\usepackage{mathrsfs}

\newcommand{\R}{{\mathbb R}}
\newcommand{\N}{{\mathbb N}}

\renewcommand{\S}{{\mathbb S}}

\newcommand{\K}{\mathsf K^\star_{\alpha,n,p}}
\newcommand{\DD}{\mathsf D_\alpha\kern0.5pt}
\newcommand{\DDstar}{\mathsf D_\alpha^*\kern0.5pt}
\newcommand{\p}{\mathsf p}
\newcommand{\rs}{\mathsf r}
\newcommand{\be}[1]{\begin{equation}\label{#1}}
\newcommand{\ee}{\end{equation}}
\renewcommand{\(}{\left(}
\renewcommand{\)}{\right)}
\newcommand{\ird}[1]{\int_{\R^d}{#1}\,dx}
\newcommand{\nrm}[2]{\left\|{#1}\right\|_{#2}}

\newcommand{\D}{\mathsf D_\alpha\kern0.5pt}
\newcommand{\Dstar}{\mathsf D_\alpha^*\kern1pt}

\newcommand{\F}{{\mathsf F}}
\newcommand{\sphere}{{\mathbb S^{d-1}}}

\newcommand{\iwrd}[1]{\int_{\R^d}{#1}\,d\mu_n}
\newcommand{\isph}[1]{\int_{\sphere}#1\,d\omega}
\newcommand{\icircle}[1]{\int_{\S^1}#1\,d\theta}
\newcommand{\msc}[1]{\href{https://mathscinet.ams.org/mathscinet/search/mscdoc.html?code=#1}{#1}}

\usepackage{etoolbox}
\newcounter{taggedeq}
\pretocmd{\equation}{\stepcounter{taggedeq}}{}{}

\newlength{\bibitemsep}
\newlength{\bibparskip}
\let\oldthebibliography\thebibliography
\renewcommand\thebibliography[1]{\oldthebibliography{#1}\setlength{\parskip}{\bibitemsep}\setlength{\itemsep}{\bibparskip}}

\begin{document}
\title{Symmetry breaking and weighted Euclidean logarithmic Sobolev inequalities}
\shorttitle{Weighted logarithmic Sobolev inequalities}
\author{Jean Dolbeault\affil1 and Andres Zuniga\affil2}
\abbrevauthor{J.~Dolbeault and A.~Zuniga}
\headabbrevauthor{Dolbeault, J. and Zuniga, A.}
\address{
\affilnum1 CEREMADE (CNRS UMR n$^\circ$ 7534), PSL university, Universit\'e Paris-Dauphine,\newline Place de Lattre de Tassigny, 75775 Paris 16, France\\
\affilnum2 Instituto de Ciencias de la Ingenier\'{\i}a (ICI), Universidad de O'Higgins (UOH), Avenida Libertador Ber\-nardo O'Higgins 611, Rancagua, Chile}
\correspdetails{dolbeaul@ceremade.dauphine.fr}
\begin{abstract}\noindent{\sc Abstract.\,}
On the Euclidean space, we establish some \emph{Weighted Logarithmic Sobolev} (WLS) inequalities. We characterize a symmetry range in which optimal functions are radially symmetric, and a symmetry breaking range. (WLS) inequalities are a limit case for a family of subcritical \emph{Caffarelli-Kohn-Nirenberg} (CKN) inequalities with similar symmetry properties. A generalized \emph{carr\'e du champ} method applies not only to the optimal solution of the nonlinear elliptic Euler-Lagrange equation and proves a rigidity result as for (CKN) inequalities, but also to entropy type estimates, with the full strength of the \emph{carr\'e du champ} method in a parabolic setting. This is a significant improvement on known results for (CKN). Finally, we briefly sketch some consequences of our results for the weighted diffusion flow.
\\[4pt]
\\[4pt]
\end{abstract}
\maketitle

\section{Introduction and main results}\label{Sec:Intro}

\emph{Logarithmic Sobolev inequalities} are well known cases of functional inequalities with many applications in various areas of mathematics ranging from information theory to probability theory, functional analysis, differential geometry and mathematical physics. In partial differential equations, these inequalities now appear as fundamental tools for the understanding of rates of convergence, not only for diffusion equations but also, for instance, in kinetic theory. Various settings have been considered depending on the geometry, the presence of a drift or a potential, or the choice of a reference measure. Sharp inequalities, with optimal constants, and equality cases are trickier issues, as the problem is usually difficult to reduce to spectral estimates. Among the few known examples, we can quote the case of the sphere and the characterization~\cite{MR1132315} by E.~Carlen of the set of optimal functions in the Euclidean logarithmic Sobolev inequalities. See~\cite{MR2366398} for general weights.

In this article, we mainly focus on the case of $\R^d$ with homogeneous (power-law) weights because of the \emph{symmetry versus symmetry breaking} issue. This is a well-known question for Caffarelli-Kohn-Nirenberg inequalities. Although the weights are invariant under rotations, optimal functions are not necessarily spherically symmetric. V.~Felli and M.~Schneider gave in~\cite{Felli2003} a condition for symmetry breaking based on the linear instability of the radial solutions of the Euler-Lagrange equations. Symmetry is a global property. Proving symmetry is therefore a delicate issue and standard methods like moving planes or symmetrization techniques are not sufficient to cover all cases. The problem has recently been fully solved in~\cite{DEL-2015,Dolbeault2017} for some special sub-families of the Caffarelli-Kohn-Nirenberg inequalities using a nonlinear version of the \emph{carr\'e du champ} method introduced by D.~Bakry and M.~Emery in~\cite{Bakry1985}, applied to the Euler-Lagrange equation solved by the optimal functions. The underlying framework is based on entropy methods for nonlinear diffusion equations, but the approach is so far formal by lack of regularity estimates to justify all computations: see~\cite{DEL-JEPE} for partial results. In the case of (WLS) inequalities, we can use the whole parabolic approach of entropy methods as there is a dense set of (Hermite) polynomials and integrations by parts can be justified. As far as we know, this is the first application of the parabolic \emph{carr\'e du champ} method to the \emph{symmetry versus symmetry breaking} issue for weighted inequalities on~$\R^d$. 

\medskip Let $\mathrm L^q_\gamma(\R^d)$ with $d\ge1$ be the space of all measurable functions~$f$ such that
\[
\nrm f{q,\gamma}:=\(\ird{|f|^q\,|x|^{-\gamma}}\)^{1/q}
\]
is finite. We also define the space $\mathrm H^1_{\beta,\gamma}(\R^d)$ of the functions $f\in\mathrm L^2_\gamma(\R^d)$ such that $\nabla f\in\mathrm L^2_\beta(\R^d)$ and consider the \emph{weighted logarithmic Sobolev inequality}
\be{WLSI}\tag{WLS}
\ird{\frac{|f|^2}{\nrm f{2,\gamma}^2}\,\log\(\frac{|f|^2}{\nrm f{2,\gamma}^2}\)\,|x|^{-\gamma}}\le\mathscr C_{\beta,\gamma}+\frac n2\,\log\(\frac{\nrm{\nabla f}{2,\beta}^2}{\nrm f{2,\gamma}^2}\)\quad\forall\,f\in\mathrm H^{1}_{\beta,\gamma}(\R^d)
\ee
with
\be{n}
n:=\frac{2\,(d-\gamma)}{\beta+2-\gamma}
\ee
and real parameters $\beta$ and $\gamma$ satisfying the condition
\be{Range}
\gamma-2<\beta<\frac{d-2}d\,\gamma<d\,.
\ee
In~\eqref{WLSI}, $\mathscr C_{\beta,\gamma}$ denotes the optimal constant. Let us define the \emph{Felli \& Schneider} curve
\be{betaFS}
\beta_{\rm FS}(\gamma):=d-2-\sqrt{(d-\gamma)^2-4\,(d-1)}\,,
\ee
consider the additional parameter
\be{alpha}
\alpha:=1+\frac{\beta-\gamma}2\,,
\ee
define the function
\[
f_\star(x):=c_{n,d}\,\sqrt\alpha\,e^{-\,\frac14\,|x|^{2\alpha}}\quad\mbox{with}\quad c_{n,d}=\sqrt{\frac{\Gamma\!\(\frac d2\)}{2^\frac n2\,\pi^\frac d2\,\Gamma\!\(\frac n2\)}}
\]
such that $\nrm{f_\star}\gamma=1$ and the constant
\be{Cstar}
\mathscr C_{\beta,\gamma}^\star:=\log\(\frac{\(\frac2{n\,e}\)^\frac n2}{\alpha^{n-1}\,\pi^\frac d2}\,\frac{\Gamma\!\(\frac d2\)}{\Gamma\!\(\frac n2\)}\)\,.
\ee
Our main result deals with the \emph{symmetry versus symmetry breaking} issue and goes as follows.
\begin{theorem}\label{Thm:WLSI1} \emph{Let $d\ge2$. Assume that $(\beta,\gamma)\neq(0,0)$ satisfies~\eqref{Range}. Then Inequality~\eqref{WLSI} holds for some constant $\mathscr C_{\beta,\gamma}\le\mathscr C_{\beta,\gamma}^\star$. Equality in~\eqref{WLSI} is achieved by an optimal function $f_{\beta,\gamma}\in\mathrm H^{1,1}_{\beta,\gamma}(\R^d)\setminus\{0\}$ and there are two cases:
\begin{itemize}
\item[\rm (i)] \emph{Symmetry breaking :} $\mathscr C_{\beta,\gamma}<\mathscr C_{\beta,\gamma}^\star$ and $f_{\beta,\gamma}$ is not radially symmetric if and only if
\be{eq:symmetrybreaking:range}
\gamma<0\quad\mbox{and}\quad\beta_{\rm FS}(\gamma)<\beta<\frac{d-2}d\,\gamma\,.
\ee
\item[\rm (ii)] \emph{Symmetry :} $\mathscr C_{\beta,\gamma}=\mathscr C_{\beta,\gamma}^\star$ and all optimal functions are given by $f_\star$ up to a multiplication by an arbitrary real constant and a scaling if and only if
\be{eq:symmetry:range}
\gamma<d\quad\mbox{and}\quad\gamma-2\le\beta\le\beta_{\rm FS}(\gamma)\,.
\ee
\end{itemize}
}\end{theorem}
\noindent In the symmetry breaking range, the set of optimal functions is generated by $f_{\beta,\gamma}$ up to rotations, a multiplication by an arbitrary real constant and a scaling. If $(\beta,\gamma)=(0,0)$, optimality is achieved by Gaussian functions corresponding to $\alpha=1$ and translations also have to be taken into account according to~\cite{MR1132315}. If $d=1$, we have the same result as in the symmetry case: $\mathscr C_{\beta,\gamma}=\mathscr C_{\beta,\gamma}^\star$ and all optimal functions are given by $f_\star$ up to a multiplication by an arbitrary real constant and a scaling. At this stage, the driving mechanism responsible for the symmetry breaking phenomenon might still look somewhat mysterious. We will now reformulate Inequality~\eqref{WLSI} into various equivalent forms before coming back to a qualitative explanation of the competition between terms of different nature which explains why symmetry breaking occurs in the range $\beta>\beta_{\rm FS}(\gamma)$.

\medskip As in~\cite{MR3579563}, we can reduce~\eqref{WLSI} to the case $\beta=\gamma$, at the price of an anisotropy in the gradient term measured by $\alpha\neq1$. Let us consider the \emph{artificial dimension} $n$ given by~\eqref{n} and take
\[
\nu:=d-n\,.
\]
Let us define the operator $\D$
\[
\D=\nabla+(\alpha-1)\,\frac x{|x|^2}\,(x\cdot\nabla)=\nabla+(\alpha-1)\,\omega\,\partial_r\,,
\]
so that, in spherical coordinates $(r,\omega)\in\R^+\times\S^{d-1}$, it writes
\[
\DD w=\begin{pmatrix}\alpha\,\partial_r w\\ \frac1r\nabla_\omega w\end{pmatrix}\,.
\]
By Condition~\eqref{Range}, notice that $n>d$ and $\nu<0$ arise from $\beta<(d-2)\,\gamma/d$.

To a function $f\in\mathrm H^1_{\beta,\gamma}(\R^d)$, we associate the function $g\in\mathrm H^1_{\nu,\nu}(\R^d)$ such that
\be{ChangeOfDimension}
f(x)=g\(|x|^{\alpha-1}\,x\)\quad\forall\,x\in\R^d\,.
\ee
With this change of variables, the function $f_\star$ is transformed into the \emph{Gaussian} function
\be{Gaussian}
g_\star(x)=c_{n,d}\,e^{-\,\frac14\,|x|^2}\quad\mbox{with}\quad c_{n,d}=\sqrt{\frac{\Gamma\!\(\frac d2\)}{2^\frac n2\,\pi^\frac d2\,\Gamma\!\(\frac n2\)}}\,,
\ee
where the normalization constant $c_{n,d}$ is such that $\nrm{g_\star}{2,\nu}=1$. Let us define
\[
\alpha_{\mathrm{FS}}:=\sqrt{\frac{d-1}{n-1}}\quad\mbox{and}\quad\mathscr K_{n,\alpha}^\star:=\mathscr C_{\beta,\gamma}^\star-\log\alpha=\log\(\frac{\(\frac2{n\,e}\)^\frac n2}{\alpha^n\,\pi^\frac d2}\,\frac{\Gamma\!\(\frac d2\)}{\Gamma\!\(\frac n2\)}\)\,.
\]
Using~\eqref{ChangeOfDimension}, Inequality~\eqref{WLSI} is transformed into the $n$-\emph{dimensional weighted logarithmic Sobolev inequality}
\be{WLSI-alpha}\tag{WLS$_n$}
\ird{\frac{|g|^2}{\nrm g{2,\nu}^2}\,\log\(\frac{|g|^2}{\nrm g{2,\nu}^2}\)\,|x|^{-\nu}}\le\mathscr K_{n,\alpha}+\frac n2\,\log\(\frac{\nrm{\,\DD g}{2,\nu}^2}{\nrm g{2,\nu}^2}\)
\ee
where $n$ plays the role of a dimension at least for scaling properties, even if it is not an integer. Rewritten with the parameters $\alpha$, $n$ and $\nu$, Theorem~\ref{Thm:WLSI1} goes as follows.
\begin{corollary}\label{Cor:WLSI1} \emph{Let $n>d\ge1$, $\nu=d-n$, and assume that $\alpha\in(0,1)\cup(1,+\infty)$. Then Inequality~\eqref{WLSI-alpha} holds for some constant $\mathscr K_{n,\alpha}\le\mathscr K_{n,\alpha}^\star$. Equality in~\eqref{WLSI-alpha} is achieved by an optimal function \hbox{$g_{\alpha,n}\in\mathrm H^{1}_{\nu,\nu}(\R^d)\setminus\{0\}$} and there are two cases:
\begin{itemize}
\item[\rm (i)] \emph{Symmetry breaking :} $\mathscr K_{n,\alpha}<\mathscr K_{n,\alpha}^\star$ and $g_{\alpha,n}$ is not radially symmetric if and only if $\alpha>\alpha_{\rm{FS}}$ and $d\ge2$.
\item[\rm (ii)] \emph{Symmetry :} $\mathscr K_{n,\alpha}=\mathscr K_{n,\alpha}^\star$ and all optimal functions are given by $g_\star$ up to a multiplication by an arbitrary real constant and a scaling if and only if either $d\ge2$ and $\alpha\le\alpha_{\rm{FS}}$, or $d=1$.\end{itemize}
}\end{corollary}
\noindent Since Corollary~\ref{Cor:WLSI1} is equivalent to Theorem~\ref{Thm:WLSI1} by the change of variables~\eqref{ChangeOfDimension}, we will use interchangeably the two statements. Any result proved for~\eqref{WLSI} is also true for~\eqref{WLSI-alpha} and reciprocally. There are various other equivalent forms of the~\eqref{WLSI} inequalities, exactly as for the standard logarithmic Sobolev inequalities:
\begin{itemize}
\item[$\rhd$] The two \emph{non-scale invariant weighted logarithmic Sobolev inequalities},
\begin{subequations}
\begin{align}
&\nrm{\nabla f}{2,\beta}^2-\sigma\ird{|f|^2\,\log\(\frac{|f|^2}{\nrm f{2,\gamma}^2}\)\,|x|^{-\gamma}}\ge\sigma\(\frac n2\,\log\(\frac{2\,e}{n\,\sigma}\)-\mathscr C_{\beta,\gamma}\)\nrm f{2,\gamma}^2\,,\label{WLSI-unscaled}\\
&\nrm{\,\DD g}{2,\nu}^2-\sigma\ird{|g|^2\,\log\(\frac{|g|^2}{\nrm g{2,\nu}^2}\)\,|x|^{-\nu}}\ge\sigma\(\frac n2\,\log\(\frac{2\,e}{n\,\sigma}\)-\mathscr K_{n,\alpha}\)\nrm g{2,\nu}^2\,,\label{WLSI-alpha-unscaled}
\end{align}
\end{subequations}
hold for any $\sigma>0$ and are equivalent to~\eqref{WLSI} and~\eqref{WLSI-alpha}. This can be recovered by optimizing the left-hand sides under the scalings $\lambda\mapsto\lambda^{(d-\gamma)/2}\,f(\lambda\,\cdot)$ and $\lambda\mapsto\lambda^{n/2}\,g(\lambda\,\cdot)$. The equality case in~\eqref{WLSI-alpha-unscaled} is achieved by the function $g_\star^{\alpha,\sigma}(x):=\big(2\,\sigma\,\alpha^{-2}\big)^{ n/4}\,g_\star\big(\sqrt{2\,\sigma}\,x/\alpha\big)$ if $\mathscr K_{n,\alpha}=\mathscr K_{n,\alpha}^\star$ (symmetry case). Here $g_{\star}$ is the Gaussian function given in~\eqref{Gaussian}.
\item[$\rhd$] \emph{Gaussian-like inequalities}. In the case of the standard Sobolev inequality, without weights, the Euclidean form of the inequality is equivalent to the Gaussian form. We have the exact counterpart, which goes as follows. Let us define the probability measure
\[
d\nu_\sigma:=\nu_\sigma\,dx\quad\mbox{with}\quad\nu_\sigma(x):=|x|^{-\nu}\,\big(g_\star^{\alpha,\sigma}(x)\big)^2=c_{n,d}^2\,\big(\tfrac{2\,\sigma}{\alpha^2}\big)^\frac n2\,|x|^{-\nu}\,e^{-\,\frac\sigma{\alpha^2}\,|x|^2}
\]
with $g_\star^{\alpha,\sigma}$ defined as above. Then~\eqref{WLSI-alpha-unscaled} applied to the function $g=v\,g_\star^{\alpha,\sigma}$ amounts~to
\begin{subequations}
\be{Gaussian-2}
\int_{\R^d}|\,\DD v|^2\,d\nu_\sigma\ge\sigma\int_{\R^d}|v|^2\,\log\(\frac{|v|^2}{\int_{\R^d}|v|^2\,d\nu_\sigma}\)\,d\nu_\sigma+\sigma\,\Big(\mathscr K_{n,\alpha}-\mathscr K_{n,\alpha}^\star\Big)\int_{\R^d}|v|^2\,d\nu_\sigma
\ee
for any $v\in\mathrm H^1(\R^d,d\nu_\sigma)$, with $\mathscr K_{n,\alpha}=\mathscr K_{n,\alpha}^\star$ in the symmetry range and $\mathscr K_{n,\alpha}-\mathscr K_{n,\alpha}^\star<0$ in the symmetry breaking range. Using the change of variable~\eqref{ChangeOfDimension} with $u(x)=v\(|x|^{\alpha-1}\,x\)$ and the probability measure
\[
d\mu_\sigma:=\mu_\sigma\,dx\quad\mbox{with}\quad\mu_\sigma(x):=\alpha\,\nu_\sigma\(|x|^{\alpha-1}\,x\),
\]
we also obtain
\be{Gaussian-1}
\int_{\R^d}|\nabla u|^2\,|x|^{\gamma-\beta}\,d\mu_\sigma\ge\sigma\int_{\R^d}|u|^2\,\log\(\frac{|u|^2}{\int_{\R^d}|u|^2\,\,d\mu_\sigma}\)\,d\mu_\sigma+\sigma\,\Big(\mathscr C_{\beta,\gamma}-\mathscr C_{\beta,\gamma}^\star\Big)\int_{\R^d}|u|^2\,d\mu_\sigma\,.
\ee
If $\mathscr C_{\beta,\gamma}=\mathscr C_{\beta,\gamma}^\star$ (or equivalently $\mathscr K_{n,\alpha}=\mathscr K_{n,\alpha}^\star$), the equality case in~\eqref{Gaussian-1} is achieved by the function $u(x)=1$ a.e. and in~\eqref{Gaussian-2} by $v(x)=1$ a.e.
\end{subequations}
\item[$\rhd$] \emph{Euclidean logarithmic Sobolev inequalities with Hardy-type correction terms}. We denote by $\mathrm L^q(\R^d)$ the standard Lebesgue space with norm $\nrm fq:=\nrm f{q,0}$ and consider the function $h(x):=|x|^{-\nu/2}\,g(x)$. An expansion of the square and an integration by parts show that
\[
\nrm{\,\DD g}{2,\nu}^2=\nrm{\,\DD h+\frac12\,\alpha\,\nu\,\frac x{|x|^2}\,h}2^2=\nrm{\,\DD h}2^2-\frac14\,\alpha^2\,\nu\,\big(2\,(d-2)-\nu\big)\ird{\frac{|h|^2}{|x|^2}}\,.
\]
For any $\sigma>0$, we can rewrite~\eqref{WLSI-alpha-unscaled} in terms of $h$ as
\be{WLSI-Schr}
\ird{\(|\,\DD h|^2+V_{\alpha,\nu,\sigma}\,|h|^2-\sigma\,|h|^2\,\log\(\tfrac{|h|^2}{\nrm h2^2}\)\)}\ge\sigma\(\tfrac n2\,\log\(\tfrac{2\,e}{n\,\sigma}\)-\mathscr K_{n,\alpha}\)\nrm h2^2\quad\forall\,h\in\mathrm H^1(\R^d)
\ee
where the left-hand side is a Schr\"odinger energy with an anisotropic kinetic term if $\alpha\neq1$, a logarithmic nonlinearity and a potential
\be{potential}
V_{\alpha,\nu,\sigma}(x):=-\frac14\,\alpha^2\,\nu\,\big(2\,(d-2)-\nu\big)\,\frac1{|x|^2}-\sigma\,\nu\,\log|x|\quad\forall\,x\in\R^d\setminus\{0\}\,.
\ee
We recall that $\nu$ is a negative parameter: the potential $V_{\alpha,\nu,\sigma}$ is radially symmetric, with a singularity at $x=0$ such that $\lim_{x\to0}V_{\alpha,\nu,\sigma}(x)=+\infty$, and we also have $\lim_{|x|\to+\infty}V_{\alpha,\nu,\sigma}(x)=+\infty$. An elementary computation shows that $V_{\alpha,\nu,\sigma}$ achieves its minimum on the centered sphere of radius $\alpha\,\sqrt{(2\,(d-2)-\nu)/(2\,\sigma)}$.
\end{itemize}

Inequality~\eqref{WLSI-Schr} is typical a problem for symmetry breaking. If $\alpha=1$ and we omit the potential, a Schwarz symmetrization shows that the minimum of the Schr\"odinger energy is achieved by a radially symmetric function~$h$, up to a translation. On the other hand, if we include $V_{\alpha,\nu,\sigma}$, in order to minimize the potential energy term, it is favourable to localize as much as possible $h$ close to a point $\bar x$ in the set of the minima of~$V_{\alpha,\nu,\sigma}$ which, as a curved surface is not compatible with radial symmetry centred at $\bar x$. A competition between the kinetic and the potential energy terms is taking place, and the arbitrage is gauged by the parameter $\alpha$. In the limiting case $\alpha\to0$, only gradients in the angular variables have to be minimized to make $\nrm{\,\DD h}2^2$ small, which favours radially symmetric solutions, but this is not the case for $\alpha$ large. Going further in this qualitative analysis to decide which one of the two terms wins is difficult. It is the purpose of this paper to give a clear cut~answer.

\medskip The \emph{carr\'e du champ} method does not only determine the optimal functions in the weighted logarithmic Sobolev inequalities but also characterizes all positive critical points. If $\sigma=1$, let us consider the Euler-Lagrange equations associated with~\eqref{WLSI-unscaled} and~\eqref{WLSI-alpha-unscaled}, that is,
\be{EL}
-\,|x|^\gamma\,\nabla\cdot\(|x|^{-\beta}\,\nabla f\)+f=f\,\log|f|^2\quad\mbox{and}\quad\DDstar\,\DD\,g+g=g\,\log|g|^2
\ee
for an appropriate choice of $\nrm f{2,\gamma}^2$ and $\nrm g{2,\nu}^2$. We have the following \emph{rigidity} result.
\begin{corollary}\label{Cor:Rigidity} \emph{Under the assumptions of Theorem~\ref{Thm:WLSI1} or Corollary~\ref{Cor:WLSI1}, each of the two equations of~\eqref{EL} admits a unique positive solution, given respectively, up to a scaling and a multiplication by a positive constant, by $f_\star$ and $g_\star$, in the \emph{symmetry} range. In the \emph{symmetry breaking} range, each of the two equations admits at least one radially symmetric solution and a continuum of no-radial solutions.
}\end{corollary}
In the limit case $(\beta,\gamma)=(0,0)$ corresponding to $(n,\alpha)=(d,1)$, which is not covered in~\eqref{Range}, uniqueness is achieved only up to additional translations. We will not give a detailed proof of Corollary~\ref{Cor:Rigidity}, as it is an elementary consequence of the proof of Corollary~\ref{Cor:WLSI1}. From the point of view of nonlinear elliptic equations, it amounts to test the equations of~\eqref{EL} by $-\,|x|^\gamma\,\nabla\cdot\(|x|^{-\beta}\,\nabla f\)$. To implement the \emph{carr\'e du champ} method, we use a dynamical version of these test functions given by the weighted heat flows
\be{heat}
\frac{\partial u}{\partial t}=|x|^\gamma\,\nabla\cdot\(|x|^{-\beta}\,\nabla u\)\,,
\ee
Proving~\eqref{WLSI} and~\eqref{WLSI-alpha} in the symmetry range is obtained by identifying the optimal decay rate of the \emph{entropy}. The core of the method of D.~Bakry and M.~Emery is to evolve the entropy by the weighted heat flow: its time-derivative is the \emph{Fisher information}. Reapplying the flow, the key point is to prove the exponential decay of the Fisher information by computing one more $t$-derivative.

\medskip Let us give a a brief review of the literature. For sake of simplicity, results involving powers of $|\nabla f|^p$ with $p\neq2$, higher order derivatives related for instance to Rellich inequalities, critical weights corresponding to Hardy-type inequalities or results on general manifolds or on Lie groups will not be systematically mentioned, but we will give at least some entry points in the literature. \emph{Logarithmic Sobolev inequalities} have been widely studied, in various settings: see~~\cite{Gross75,Federbush,MR823597,MR0109101,MR479373} for historical references,~\cite{MR2352327,Guionnet-Zegarlinski03,MR1845806} for introductory books or lecture notes, and~\cite[Chapter~5]{MR3155209} for a general presentation of $\mathrm{CD}(\rho,N)$ methods applied to functional inequalities. In~\cite{MR2198019,MR2351133,MR2320410} and~\cite[Chapter~5]{Wang:1250982}, one can find various sufficient conditions for logarithmic Sobolev inequalities to hold. See~\cite{MR1796718,MR2320410},~\cite[Chapter~6]{Wang:1250982} and~\cite{MR4372142} for some results on the interpolation inequalities between Poincar\'e and logarithmic Sobolev inequalities. Optimal constants and equality cases in logarithmic Sobolev inequalities are tricky issues: beyond observations based on the \emph{carr\'e du champ} in~\cite{Bakry1985}, we refer to~\cite{MR1132315} in the Euclidean and Gaussian cases, to~\cite{DE2010} on cylinders (in connection with Caffarelli-Kohn-Nirenberg inequalities), and to~\cite{BDS} for recent considerations on stability in strong norms (see references therein for other stability results measured in, \emph{e.g.}, Wasserstein distance).

In this paper we consider the simple setting of $\R^d$ with power-law weights, for scaling reasons. Norms other than the standard Euclidean norm could be considered, but the corresponding symmetry results are, to the best of our knowledge, unknown. Our~\eqref{WLSI} inequalities appear as a limit case for a family of subcritical Caffarelli-Kohn-Nirenberg inequalities~\eqref{CKN}, for which \emph{symmetry breaking} is a well known issue that was addressed in various papers: see~\cite{Chou-Chu-93,MR1731336,MR1132797,MR2001882,Catrina2001,Felli2003,DELT09,MR3579563,Bonforte201761,DEL-2015,Dolbeault2017}, among others. See Section~\ref{Sec:CKN} for some explanations of the mathematical issues. Concerning positive critical points of~\eqref{CKN}, a rigidity result holds as a consequence of a generalized \emph{carr\'e du champ} method applied to the nonlinear elliptic equation solved by the optimal functions. This rigidity result can be rephrased in terms of the properties of branches of solutions of nonlinear elliptic equations depending on a parameter: see~\cite{1703}. At a formal level, these results for~\eqref{CKN} can also be interpreted in the framework of entropy methods as strict monotonicity properties deduced from the \emph{carr\'e du champ} method adapted to nonlinear diffusion equations.

The results on sharp functional inequalities in~\cite{Dolbeault20141338,DEL-2015} are inspired, on the one hand, by the rigidity results for nonlinear elliptic equations studied in~\cite{BV-V,MR615628}, and on the other hand, by entropy and diffusion flows of~\cite{Bakry1985,MR2381156}. The connection is made precise and expanded in~\cite{Dolbeault20141338,1504}: the \emph{carr\'e du champ} method is a central idea for the overall strategy which applies very well to linear diffusion flows with drift potential terms or on compact manifolds. The \emph{carr\'e du champ} method has many aspects, but from the functional inequalities point of view, one can just keep in mind that monotonicity properties through the diffusion flow relate the functional inequality written for an arbitrary initial data to an asymptotic regime, which can be studied using spectral methods. See~\cite{BDNS2021} for an extended presentation applied to a family of Gagliardo-Nirenberg-Sobolev inequalities. Applied to nonlinear flows on the Euclidean space, new difficulties arise as, for instance, integration by parts require precise decay bounds which are not easy to justify. Progress in the absence of singular weights has been achieved in~\cite{Carrillo2001,MR3497125}. In presence of weights, the method formally applies but only partial results have been rigorously justified in~\cite{DEL-JEPE,dolbeault2022parabolic,bonforte:hal-03581542}. For the optimizers of functional inequalities involving singular weights, the difficulty can be bypassed by proving the existence of minimizers and testing directly the solutions of the Euler-Lagrange, which amounts to testing such critical points in the direction corresponding to the flow. This is the simplest interpretation of the method of B.~Gidas and J.~Spruck in~\cite{MR615628}. The issue is then reduced to a rigidity issue for solutions of elliptic equations which, as such, have good regularity and decay properties. So far, all sharp results of symmetry in Caffarelli-Kohn-Nirenberg inequalities have been obtained using such an approach. In the case of logarithmic Sobolev inequalities, we are able to perform the whole parabolic method as there is a dense set of Hermite functions, in the appropriate version of the inequality, and integrations by parts can be justified. To our knowledge, this is the first result of symmetry \emph{versus} symmetry breaking to be proved with the parabolic version of the \emph{carr\'e du champ} method.

The logarithmic Hardy inequalities studied in~\cite{delPino20102045,DoEsFiTe2014} correspond to a boundary of the admissible domain of parameters in~\eqref{WLSI}. So far, we are not aware of a method that would allow us to deduce results from~\eqref{WLSI} by taking an appropriate limit. For completeness, let us give a few additional reading indications on papers related with ours. Concerning logarithmic Hardy and Sobolev inequalities on Lie groups, we refer to~\cite{chatzakou2021logarithmic} and references therein. See~\cite{MR4186674} and references therein for logarithmic inequalities involving powers of $|\nabla f|^p$ with $p\neq2$. We refer to~\cite{Bobkov_2009,MR2609591,MR3008255} for logarithmic Sobolev inequalities corresponding to non-singular weights of the form $(1+|x|^2)^{-\beta/2}$ known as Cauchy measures and their links with the \emph{super Poincar\'e inequalities}, and to~\cite{MR3603301,MR3647065,MR4013832} for various related contributions.

\medskip This paper is organized as follows. In Section~\ref{Sec:WLS}, we use the spectral method of V.~Felli and M.~Schneider to prove the linear instability of the radial optimal functions in the symmetry breaking range of~\eqref{WLSI} and the \emph{carr\'e du champ} method to establish the symmetry in the symmetry range, with self-contained proofs. We use entropy methods in a parabolic setting to prove the symmetry result, which is the first result of this type obtained at non-formal level using a diffusion equation of evolution, in presence of weights. In Section~\ref{Sec:CKN}, we show how~\eqref{WLSI} can be seen as a limit case for a family of subcritical Caffarelli-Kohn-Nirenberg inequalities. Notice that $\Gamma$-convergence methods is expected to provide us with an alternative proof of Corollary~\ref{Cor:WLSI1} and Theorem~\ref{Thm:WLSI1}. Section~\ref{Sec:Flows} is devoted to some consequences of our results for the weighted heat flow associated to our weighted logarithmic Sobolev inequalities.

\section{Optimal functions, symmetry and symmetry breaking in \texorpdfstring{\eqref{WLSI}}{WLSI} inequalities}\label{Sec:WLS}

\subsection{The weighted logarithmic Sobolev inequality}

We start by proving that Inequality~\eqref{WLSI} is well-defined.
\begin{lemma}\label{Lem:WLSI} \emph{Let $d\ge1$. Assume that $(\beta,\gamma)\neq(0,0)$ satisfies~\eqref{Range}. Then the inequality
\[
\ird{\frac{|f|^2}{\nrm f{2,\gamma}^2}\,\log\(\frac{|f|^2}{\nrm f{2,\gamma}^2}\)\,|x|^{-\gamma}}\le\mathscr C_{\beta,\gamma}^\star+\frac n2\,\log\(\frac{\nrm{\nabla f}{2,\beta}^2}{\nrm f{2,\gamma}^2}\)\quad\forall\,f\in\mathrm H^{1,1}_{\beta,\gamma}(\R^d)
\]
holds with $\mathscr C_{\beta,\gamma}^\star$ defined by~\eqref{Cstar}.
}\end{lemma}
\noindent In other words, we prove here that Inequality~\eqref{WLSI} holds for some constant $\mathscr C_{\beta,\gamma}\le\mathscr C_{\beta,\gamma}^\star$. Since~\eqref{WLSI} is subcritical, it is a standard strategy to establish the inequality using an H\"older interpolation and a critical inequality.
\begin{proof} Let $p:=2\,\frac{d-\gamma}{d-2-\beta}\in(2,2^*)$ with $2^*=+\infty$ if $d=1$, $2$ and $2^*=2\,d/(d-2)$ if $d\ge3$. Let us consider the critical Caffarelli-Kohn-Nirenberg inequality
\be{CKNcrit}
\ird{|\nabla u|^2\,|x|^{-\beta}}\ge\mathcal C_{\rm{CKN}}\(\ird{|u|^p\,|x|^{-\gamma}}\)^{2/p}
\ee
which has been widely studied, see for instance~\cite{Ilyin,Caffarelli1984,Catrina2001,DEL-2015}. Here $\beta$ and $\gamma$ satisfy~\eqref{Range} and $n$ given by~\eqref{n} is such that
\[
\frac p{p-2}=\frac n2\,.
\]
H\"older's inequality
\[
\nrm u{q,\gamma}\le\nrm u2^\eta\,\nrm u{p,\gamma}^{1-\eta}\,,
\]
written with $\eta=2\,\frac{p-q}{q\,(p-2)}$ for any $q\in(2,p)$, degenenerates into an equality as $q\to2_+$. By differentiating this inequality with respect to $q$ at $q=2$, we obtain the \emph{logarithmic H\"older inequality}
\be{logHolder}
\ird{|u|^2\,\log\(\frac{|u|^2}{\nrm u{2,\gamma}^2}\)|x|^{-\gamma}}\le\frac p{p-2}\,\nrm u{2,\gamma}^2\,\log\(\frac{\nrm u{p,\gamma}^2}{\nrm u{2,\gamma}^2}\)
\ee
for any $p>2$. Combined with~\eqref{CKNcrit} in the case $p=2\,\frac{d-\gamma}{d-2-\beta}$, this establishes the \emph{weighted logarithmic Sobolev inequality}
\be{LSI}
\ird{|u|^2\,\log\(\frac{|u|^2}{\nrm u{2,\gamma}^2}\)|x|^{-\gamma}}\le\mathcal A\,\nrm u{2,\gamma}^2\,\log\(\frac{\nrm{\nabla u}{2,\beta}^2}{\nrm u{2,\gamma}^2}\)+\mathcal B\,\nrm u{2,\gamma}^2
\ee
with $\mathcal A=n/2$, $n$ given by~\eqref{n} and $\mathcal B=n\,\log\mathcal C_{\rm{CKN}}$. The value of $\mathcal A$ cannot be improved, as shown by the scaling
\[\label{Scaling}
\lambda\mapsto u_\lambda:=\lambda^{n/2}\,u(\lambda\cdot)\,.
\]
Testing~\eqref{LSI} by $f_\star(x)=c_{n,d}\,\sqrt\alpha\,e^{-\,\frac14\,|x|^{2\alpha}}$ shows that $\mathcal B\ge\mathscr C_{\beta,\gamma}^\star$. The optimal value of $\mathcal B$ is therefore the minimal value for which~\eqref{LSI} holds for any $u\in\mathrm H^1_{\beta,\gamma}(\R^d)$.
\end{proof}

\subsection{Existence of optimal functions}\label{Sec:optimality}

The existence of an optimal function for~\eqref{WLSI} is proved in~\cite{BDS} by concentration-compactness methods when $n=d$ and $\alpha=1$. A similar proof can be found in~\cite{DE2010}, which itself relies on an extension of the concentration-compactness method of~\cite{zbMATH04155282}. The proof in the case $(n,\alpha)\neq(d,1)$ can also be done by the same method. 
\begin{proposition}\label{Prop:WLSIbyCC} \emph{Let $d\ge1$ and asume that $(\beta,\gamma)$ satisfies~\eqref{Range}. Equality in~\eqref{WLSI} is achieved by a function $u\in\mathrm H^1_{\beta,\gamma}(\R^d)$ if $\mathscr C_{\beta,\gamma}$ is taken to its optimal value.}\end{proposition}
\begin{proof} We work with the inequality written in the form~\eqref{WLSI-Schr} with $V_{\alpha,\nu,\sigma}$ defined by~\eqref{potential} and rely on direct variational methods. Since $V_{\alpha,\nu,\sigma}$ is bounded from below, there is no significant difficulty compared to the proof of the existence of a minimizer for logarithmic Sobolev inequalities without weights or potentials (see for instance~\cite{BDS,DE2010} for similar results). For completeness, let us give a sketch of a proof. 

Using the homogeneity, let us consider a minimizing sequence $(h_n)_{n\in\N}$ of functions in $\mathrm H^1(\R^d)$ such that $\nrm{h_n}2=1$ for any $n\in\N$ and
\[
\lim_{n\to+\infty}\ird{\(|\,\DD h_n|^2+V_{\alpha,\nu,\sigma}\,|h_n|^2-\sigma\,|h_n|^2\,\log|h_n|^2\)}=\sigma\(\mathscr K_{n,\alpha}+\frac n2\,\log\(\frac{2\,e}{n\,\sigma}\)\)\,.
\]
An optimization under scalings shows that we can choose $\nrm{\,\DD h_n}2=\sigma\,n/2$ with no loss of generality. Using $\nrm{\,\DD h_n}2\ge\min\{1,\alpha\}\,\nrm{\nabla h_n}2$ and the standard Euclidean logarithmic Sobolev inequality, we have that $\big(|h_n|^2\,\log|h_n|^2\big)_{n\in\N}$ and $\big((V_{\alpha,\nu,\sigma})_+^{1/2}\,h_n)_{n\in\N}$ are bounded in $\mathrm L^2(\R^d)$. For any $R>1$ large enough, since
\[
\int_{|x|>R}|h_n|^2\,dx\le\frac1{\log R}\le C\ird{V_{\alpha,\nu,\sigma}\,|h_n|^2}
\]
for some positive constant $C$ and since concentration is forbidden away from origin by standard Gagliardo-Nirenberg embedding inequalities and~\eqref{logHolder} while concentration at $x=0$ would provide us with an infinite contribution to the potential energy term, the sequence $(h_n)_{n\in\N}$ is relatively compact in $\mathrm L^2(\R^d)$. Up to the extraction of a subsequence, $(h_n)_{n\in\N}$ strongly converges in $\mathrm L^2(\R^d)$ to some limit $h$ such that $\nrm h2=1$. According to~\cite[Theorem~2]{zbMATH03834677}, we have
\[
\lim_{n\to+\infty}\ird{|h_n|^2\,\log|h_n|^2}=\ird{|h|^2\,\log|h|^2}+\lim_{n\to+\infty}\ird{|h-h_n|^2\,\log|h-h_n|^2}\,.
\]
By~\eqref{WLSI-Schr} applied to $(h-h_n)$ and a convexity argument as in~\cite{DE2010}, \hbox{$\lim_{n\to+\infty}\ird{|h-h_n|^2\,\log|h-h_n|^2}=0$} and we conclude that $h$ realizes the equality case in~\eqref{WLSI-Schr}. This completes the proof.\end{proof}

\subsection{Linear instability and a symmetry breaking range}\label{Sec:LinearInstability}

With $\mathsf f(\mathsf x,\mathsf y):=\mathsf x^{1-\frac2n}\,e^{\frac2n\,\frac{\mathsf y}{\mathsf x}}$, Inequality~\eqref{WLSI-alpha} becomes
\[
\mathcal F[g]:=\nrm{\,\DD g}{2,\nu}^2-e^{-\frac2n\,\mathscr K_{n,\alpha}}\,\mathsf f\(\ird{|g|^2\,|x|^{-\nu}},\ird{|g|^2\,\log\(|g|^2\)\,|x|^{-\nu}}\)\ge0\,.
\]
We Taylor expand $\mathcal F[g]$ around $g_*$ by computing $\mathsf F[\phi]:=\frac12\,\lim\limits_{\varepsilon\to0}\varepsilon^{-2}\,\mathcal F[g_*+\varepsilon\,\phi]$ and find that
\[
\mathsf F[\phi]=\nrm{\,\DD \phi}{2,\nu}^2-e^{-\frac2n\,\mathscr K_{n,\alpha}}\,\Bigg(\partial_{\mathsf x}\mathsf f(1,\mathsf y_*)\,\nrm\phi{2,\nu}^2+\partial_{\mathsf y}\mathsf f(1,\mathsf y_*)\(3\,\nrm\phi{2,\nu}^2+\ird{\log\(|g_*|^2\)\,|\phi|^2\,|x|^{-\nu}}\)\Bigg)
\]
for any $\phi$ such that $\ird{\(1,|x|^2\)g_*\,\phi\,|x|^{-\nu}}=(0,0)$, where
\[
\ird{|g_*|^2\,|x|^{-\nu}}=1\quad\mbox{and}\quad\mathsf y_*:=\ird{|g_*|^2\,\log\(|g_*|^2\)\,|x|^{-\nu}}=2\,\log c_{n,d}-\frac n2\,.
\]
In the symmetry range, we have $\mathscr K_{n,\alpha}=\mathscr K_{n,\alpha}^\star$ and
\[
e^{-\frac2n\,\mathscr K_{n,\alpha}}\,\mathsf f(1,\mathsf y_*)=e^{-\frac2n\,\mathscr K_{n,\alpha}}\,e^{\frac2n\,\mathsf y_*}=\nrm{\,\DD g_*}{2,\nu}^2=\frac n4\,\alpha^2\,.
\]
Since $\partial_{\mathsf x}\mathsf f(1,\mathsf y_*)=\(1-\frac2n-\frac2n\,\mathsf y_*\)\mathsf f(1,\mathsf y_*)$ and $\partial_{\mathsf y}\mathsf f(1,\mathsf y_*)=\frac2n\,\mathsf f(1,\mathsf y_*)$, in that case we obtain
\[
\mathsf F[\phi]=\nrm{\,\DD \phi}{2,\nu}^2+\frac14\,\alpha^2\ird{|\phi|^2\,|x|^{2+n-d}}-\frac n4\,\alpha^2\(1-\frac2n-\frac2n\,\mathsf y_*+\frac6n+\frac4n\,\log c_{n,d}\)\,,
\]
so that, under the condition $\mathscr K_{n,\alpha}=\mathscr K_{n,\alpha}^\star$, we have
\[
\mathsf F[\phi]=\nrm{\,\DD \phi}{2,\nu}^2-\alpha^2\(1+\frac n2\)\nrm\phi{2,\nu}^2+\frac14\,\alpha^2\ird{|\phi|^2\,|x|^{2+n-d}}\,.
\]
\begin{lemma} \emph{Let $n>1$ and $\alpha>0$ be two real numbers and consider any integer $d\ge2$. If $\mathscr K_{n,\alpha}=\mathscr K_{n,\alpha}^\star$, the lowest nonradial eigenmode associated with the quadratic form $\phi\mapsto\mathsf F[\phi]$ is
\be{lambda1}
\lambda_1(\alpha)=\frac\alpha2\(\sqrt{4\,(d-1)+\alpha^2\,(n-2)^2}-\alpha\,n\)\,.
\ee
}\end{lemma}
\begin{proof} We use a decomposition into spherical harmonics. Since the lowest eigenvalue of the Laplace-Beltrami operator on $\S^{d-1}$ is $(d-1)$, we have to solve the eigenvalue problem
\[
-\,\alpha^2\(\varphi''+\frac{n-1}r\,\varphi'\)+\frac{d-1}{r^2}\,\phi+\frac{\alpha^2}4\,r^2\,\varphi=\lambda_1\,\varphi
\]
for some positive radial function $r\mapsto\varphi(r)$, $r\in (0,+\infty)$. Elementary computations show that $\varphi(r)=r^{1+\delta}\,e^{-r^2/4}$ solves the equation with $\delta=\lambda_1(\alpha)/\alpha^2$ and $\lambda_1(\alpha)$ given by~\eqref{lambda1}.\end{proof}

On $\R^+$, it is an elementary computation to check that $\alpha\mapsto\lambda_1(\alpha)$ takes negative values if and only if
\be{SymBreaking}
\alpha>\alpha_{\mathrm{FS}}=\sqrt{\frac{d-1}{n-1}}\,.
\ee
Here we find exactly the \emph{Felli \& Schneider condition} for symmetry breaking as \hbox{in~\cite{Felli2003,Dolbeault2017}}.
\begin{proposition}\label{Prop:FS} \emph{If~\eqref{SymBreaking} holds, then $\mathscr K_{n,\alpha}<\mathscr K_{n,\alpha}^\star$.}\end{proposition}
\begin{proof}
We argue by contradiction. If $\mathscr K_{n,\alpha}=\mathscr K_{n,\alpha}^\star$, then $g_\star$ is an optimal function. However, a perturgation of $g_\star$ by an eigenfunction associated with the eigenvalue $\lambda_1(\alpha)<0$ given by~\eqref{lambda1} proves that $\mathcal F$ takes negative values, a contradiction with the definition of $\mathscr K_{n,\alpha}$.
\end{proof}

We learn from Propositions~\ref{Prop:WLSIbyCC} and~\ref{Prop:FS} that Inequality~\eqref{WLSI} admits only non-radial optimal functions if~\eqref{SymBreaking} holds. The next step is to prove that all optimal functions for~\eqref{WLSI} are radially symmetric if $\alpha\le\alpha_{\mathrm{FS}}$ and given up to a multiplication by a constant and a scaling by $f_\star$.

\subsection{A symmetry result by the \texorpdfstring{\emph{carr\'e du champ}}{CarreDuChamp} method}\label{Sec:CarreDuChamp}

This section is devoted to the proof of Theorem~\ref{Thm:WLSI1}, Part (ii), corresponding to the symmetry case, which is the difficult range. The method relies on the \emph{carr\'e du champ} method and it is inspired from~\cite[Section~3]{DEL-JEPE}. However, the presence of the logarithmic nonlinearity imposes various non-trivial changes that are detailed below. Altogether, this is a striking application of the nonlinear \emph{carr\'e du champ} method and we give a complete and self-contained proof. Computations which have already appeared in the context of~\eqref{CKN} inequalities are clearly indicated. We consider here the \emph{weighted logarithmic Sobolev inequality} written in the form of~\eqref{WLSI-alpha-unscaled} with $\sigma=1/2$. If $\Dstar$ is the adjoint operator of~$\D$ acting on vector-valued functions $\F$, with respect to the measure
\[
d\mu_n:=r^{n-1}\,dr\,d\omega\,,
\]
then we have
\[
\Dstar \F=-\,|x|^{d-n}\,\nabla\cdot (|x|^{n-d}\,\F)-(\alpha-1)\,r^{1-n}\,\omega\cdot\partial_r(r^{n-1}\,\F)\,.
\]
Moreover we have the useful identity
\be{Id1}
\Dstar(u\,\F)=-\,\D u\cdot \F+u\,\Dstar \F
\ee
if $u$ and $\F$ are respectively scalar- and vector-valued functions. Let us define the operator $\mathsf L_\alpha$ by
\[
\mathsf L_\alpha=-\,\Dstar\D=\alpha^2\(\partial^2_r+\frac{n-1}r\,\partial_r\)+\frac{\Delta_\omega}{r^2}\,,
\]
where $\Delta_\omega$ denotes the Laplace-Beltrami operator on $\sphere$, and consider the Fokker-Planck equation
\be{Eqn}
\frac{\partial u}{\partial t}=-\,\Dstar(u\,\F)
\ee
where the flux $\F$ and the relative pressure variable $\mathsf p$ are defined by
\[
\F(t,x):=\D\log u+x=\D\(\log u+\frac{|x|^2}{2\,\alpha^2}\)=\D\mathsf p\,,\quad\mathsf p:=\log u-2\,\log g_\star^{\alpha,\frac12}
\]
and $g_\star^{\alpha,1/2}$ is the normalized optimal function for~\eqref{WLSI-alpha-unscaled} with $\sigma=1/2$. Hermite functions are dense and stable under the action of the flow~\eqref{Eqn} so that we can always work on a finite dimensional space generated by some Hermite functions and argue by density. Since there is no difficulty in integrating by parts, we will do it without further justification. This is the first major difference with~\cite{DEL-JEPE} where a nonlinear flow is considered and one has to do an approximation procedure on larger and larger balls. To simplify the proof, we divide it in several simple statements. Our main goal is to prove the exponential decay of the \emph{Fisher information}, which goes as follows.
\begin{lemma}\label{Lem:ExpDecay} \emph{Assume that $d\ge2$, $n>d$ and $\alpha\in(0,\alpha_{\rm{FS}}]$. If $u$ solves~\eqref{Eqn} with $\F=\D\mathsf p$, then 
\[
\frac d{dt}\iwrd{u\,|\F|^2}\le-\,2\,\alpha\iwrd{u\,|\F|^2}\,.
\]
}\end{lemma}
\noindent It is straightforward to check that
\[
\frac{\partial\F}{\partial t}+\D\,\Big(u^{-1}\,\Dstar\(u\,\F\)\Big)=0
\]
and, as a consequence,
\begin{multline*}
\frac d{dt}\iwrd{u\,|\F|^2}=\iwrd{\frac{\partial u}{\partial t}\,|\F|^2}+2\iwrd{u\,\F\cdot\frac{\partial\F}{\partial t}}\\
=-\iwrd{\Dstar(u\,\F)\,|\F|^2}-2\iwrd{u\,\F\cdot\D\,\Big(u^{-1}\,\Dstar\(u\,\F\)\Big)}\,.
\end{multline*}
A first integration by parts shows that
\[
\iwrd{\Dstar(u\,\F)\,|\F|^2}=\iwrd{u\,\F\,\cdot\D\(|\F|^2\)}\,.
\]
Using~\eqref{Id1} and $\D\log u=\F-x$, we get
\begin{eqnarray*}
\frac d{dt}\iwrd{u\,|\F|^2}&\hspace*{-5pt}=&\hspace*{-5pt}-\iwrd{u\,\F\,\cdot\D\(|\F|^2\)}-2\iwrd{u\,\F\cdot\D\(\Dstar\F-\F\cdot\D \log u\)}\\
&\hspace*{-5pt}=&\hspace*{-5pt}-\iwrd{u\,\F\,\cdot\D\(|\F|^2\)}-2\iwrd{u\,\F\cdot\D\(\Dstar\F-|\F|^2+\F\cdot x\)}\\
&\hspace*{-5pt}=&\hspace*{-5pt}\iwrd{u\,\F\,\cdot\D\(|\F|^2\)}-2\iwrd{u\,\F\cdot\D(\F\cdot x)}-2\iwrd{u\,\F\cdot\D\(\Dstar\F\)}\,.
\end{eqnarray*}
Using $u\,\F=\D u+x\,u$ and integrating by parts, we obtain
\[
\iwrd{u\,\F\,\cdot\D\(|\F|^2\)}=-\iwrd{u\,\mathsf L_\alpha\(|\F|^2\)}+\iwrd{u\,x\,\cdot\D\(|\F|^2\)}\,.
\]
Hence
\[
\frac d{dt}\iwrd{u\,|\F|^2}=-\,2\iwrd{u\,\mathsf K[\mathsf F]}+\iwrd{u\,x\,\cdot\D\(|\F|^2\)}-2\iwrd{u\,\F\cdot\D(\F\cdot x)}-\frac2n\iwrd{u\,(\Dstar\F)^2}
\]
with
\be{FK}
\mathsf K[\mathsf F]:=\frac12\,\mathsf L_\alpha\(|\F|^2\)+\F\cdot\D\(\Dstar\F\)-\frac1n\,(\Dstar\F)^2\,.
\ee
We recall that $\F=\D\mathsf p$ so that
\[
\mathsf K[\D\mathsf p]=\frac12\,\mathsf L_\alpha\,|\DD\p|^2-\,\DD\p\cdot\DD\mathsf L_\alpha\p-\frac1n\,(\mathsf L_\alpha\p)^2\,.
\]
Let us state a result inspired by~\cite[Lemma~5.1]{DEL-2015},~\cite[Lemma~4.2 and Lemma~4.3]{Dolbeault2017}.
\begin{lemma}\label{Lem:4.2} \emph{With the above notations, we have the two following estimates:
\begin{itemize}
\item[\rm (i)] Pointwise estimate:
\[
\mathsf K[\D\mathsf p]=\alpha^4\(1-\frac1n\)\left|\mathsf p''-\frac{\mathsf p'}r-\frac{\Delta_\omega\,\mathsf p}{\alpha^2\,(n-1)\,r^2}\right|^2+\frac{2\,\alpha^2}{r^2}\left|\nabla_\omega\mathsf p'-\frac{\nabla_\omega\mathsf p}r \right|^2+\frac{\mathsf k[\mathsf p]}{r^4}
\]
where
\[
\mathsf k[\mathsf p]:=\frac12\,\Delta_\omega\,|\nabla_\omega\mathsf p|^2-\nabla_\omega\mathsf p\cdot\nabla_\omega(\Delta_\omega\,\mathsf p)-\frac1{n-1}\,(\Delta_\omega\,\mathsf p)^2-(n-2)\,\alpha^2\,|\nabla_\omega\mathsf p|^2\,.
\]
\item[\rm (ii)] Integral estimate on the sphere $\mathbb S^{d-1}$ if $d\ge2$:
\[
\isph{\mathsf k[\mathsf p]\,u}\ge(n-2)\(\alpha_{\rm FS}^2-\alpha^2\)\isph{|\nabla_\omega\mathsf p|^2\,u}+\delta\isph{|\nabla_\omega\mathsf p|^4}
\]
where $\delta$ is a positive constant depending only on $n$ and $d$.
\end{itemize}
}\end{lemma}
\begin{proof} Property~(i) can be found in~\cite[Lemma~5.1]{DEL-2015} and~\cite[Lemma~4.2]{Dolbeault2017}. See Lemma~\ref{Lem:Derivmatrixform1} in~\ref{Appendix} for a more detailed statement and a proof. The regularity needed in Lemma~\ref{Lem:Derivmatrixform1} is not an issue in our setting, as we consider solutions in spaces generated by a finite number of Hermite polynomials and then argue by density.

Next we focus on the proof of (ii). We go along the lines of the proof of~\cite[Lemma~4.3]{Dolbeault2017}, but many details have to be changed to adapt the proof. The results are inspired from~\cite{MR2381156,DEKL2014,1504,Dolbeault20141338} and we adopt the presentation of~\cite{Dolbeault20141338}: $\sphere$ is considered as a $(d-1)$-dimensional compact manifold with metric~$g$ and uniform probability measure $d\omega$. We shall indeed assume that it is normalized so that $\omega(\sphere)=1$ to avoid carrying normalization constants. Let us introduce some notation. If $\mathrm A_{ij}$ and $\mathrm B_{ij}$ are two tensors, then
\[
\mathrm A:\mathrm B:=g^{im}\,g^{jn}\,\mathrm A_{ij}\,\mathrm B_{mn}\quad\mbox{and}\quad\|\mathrm A\|^2:=\mathrm A:\mathrm A\,.
\]
Here $g^{ij}$ is the inverse of the metric tensor, \emph{i.e.}, $g^{ij}\,g_{jk}=\delta^i_k$. We use the Einstein summation convention and $\delta^i_k$ denotes the Kronecker symbol. Let us denote the \emph{Hessian} by $\mathrm H_\omega\p$ and define the \emph{trace free Hessian} by
\[
\mathrm L_\omega\p:=\mathrm H_\omega\p-\frac1{d-1}\,(\Delta_\omega\p)\,g\,.
\]
We also consider the following trace free tensor
\[
\mathrm M_\omega\p:=\nabla_\omega\p\otimes\nabla_\omega\p-\frac1{d-1}\,|\nabla_\omega\p|^2\,g\,,
\]
where $\nabla_\omega\p\otimes\nabla_\omega\p:=(\partial_i\p\,\partial_j\p)_{ij}$ and $\|\nabla_\omega\p\otimes\nabla_\omega\p\|^2=|\nabla_\omega\p|^4=(g^{ij}\,\partial_i\p\,\partial_j\p)^2$. Using \hbox{$\mathrm L_\omega:g=\mathrm M_\omega:g=0$}, we obtain
\begin{subequations}\label{TraceFree}
\begin{align}
&\|\mathrm L_\omega\p\|^2=\|\mathrm H_\omega\p\|^2-\frac1{d-1}\,(\Delta_\omega\p)^2\,,\label{TraceFree1}\\
&\|\mathrm M_\omega\p\|^2=\left\|\nabla_\omega\p\otimes\nabla_\omega\p\right\|^2-\frac1{d-1}\,|\nabla_\omega\p|^4=\frac{d-2}{d-1}\,|\nabla_\omega\p|^4\,,\label{TraceFree2}\\
&\mathrm L_\omega\p:\mathrm M_\omega\p=\(\mathrm H_\omega\p-\frac1{d-1}\,(\Delta_\omega\p)\,g\):(\nabla_\omega\p\otimes\nabla_\omega\p)
=\mathrm H_\omega\p:(\nabla_\omega\p\otimes\nabla_\omega\p)-\frac1{d-1}\,\Delta_\omega\p\,|\nabla_\omega\p|^2\,.\label{TraceFree3}
\end{align}
\end{subequations}

Assume first that $d\ge3$. The Bochner-Lichnerowicz-Weitzenb\"ock formula on $\S^{d-1}$ takes the simple form
\be{BLW}
\frac12\,\Delta_\omega\(|\nabla_\omega\p|^2\)=\|\mathrm H_\omega\p\|^2+\nabla_\omega(\Delta_\omega\p)\cdot\nabla_\omega\p+(d-2)\,|\nabla_\omega\p|^2
\ee
where the last term, \emph{i.e.}, $\mathrm{Ric}(\nabla_\omega\p,\nabla_\omega\p)=(d-2)\,|\nabla_\omega\p|^2$, accounts for the Ricci curvature tensor contracted with \hbox{$\nabla_\omega\p \otimes\nabla_\omega\p$}. With
\[
\mathsf k[\p]:=\frac12\,\Delta_\omega\(|\nabla_\omega\p|^2\)-\nabla_\omega\p\cdot\nabla_\omega(\Delta_\omega\,\p)-\frac1{n-1}\,(\Delta_\omega\,\p)^2-(n-2)\,\alpha^2\,|\nabla_\omega\p|^2\,,
\]
we compute
\begin{multline*}
\isph{u\,\mathsf k[\p]}=\isph{u\(\|\mathrm H_\omega\p\|^2+(d-2)\,|\nabla_\omega\p|^2-\frac1{n-1}\,(\Delta_\omega\,\p)^2-(n-2)\,\alpha^2\,|\nabla_\omega\p|^2\)}\\
=\isph{u\(\|\mathrm L_\omega\p\|^2+\(\frac1{d-1}-\frac1{n-1}\)(\Delta_\omega\,\p)^2+\((d-2)-(n-2)\,\alpha^2\)|\nabla_\omega\p|^2\)}
\end{multline*}
using~\eqref{BLW} and~\eqref{TraceFree1}.

With $\p=\log u+\frac{|x|^2}{2\,\alpha^2}$, it turns out that
\[
\nabla_\omega u=u\,\nabla_\omega\p\,,
\]
whence, applying integrations by parts again and taking into account~\eqref{TraceFree2} and~\eqref{TraceFree3},
\begin{align*}
\isph{u\,\Delta_\omega\p\,|\nabla_\omega\p|^2}=&\,-\isph{u\,|\nabla_\omega\p|^4}-2\isph{u\,\mathrm H_\omega\p:(\nabla_\omega\p\otimes\nabla_\omega\p)}\\
=&-\,\frac{d-1}{d-2}\isph{u\,\|\mathrm M_\omega\p\|^2}-2\isph{u\,\mathrm L_\omega\p:\mathrm M_\omega\p}-\frac2{d-1}\isph{u\,\Delta_\omega\p\,|\nabla_\omega\p|^2}\,.
\end{align*}
As a consequence, we obtain
\be{Sphere:First}
\isph{u\,\Delta_\omega\p\,|\nabla_\omega\p|^2}=-\frac{d-1}{d+1}\(\frac{d-1}{d-2}\isph{u\,\|\mathrm M_\omega\p\|^2}+2\isph{u\,\mathrm L_\omega\p:\mathrm M_\omega\p}\)\,.
\ee
On the other hand, integrating~\eqref{BLW} on $\S^{d-1}$ against $u$ and performing an integration by parts shows that
\begin{multline*}
\frac12\isph{u\,\Delta_\omega\(|\nabla_\omega\p|^2\)}+\isph{u\,\Delta_\omega\p\,|\nabla_\omega\p|^2}+\isph{u\,(\Delta_\omega\p)^2}\\
=\isph{u\,\|\mathrm H_\omega\p\|^2}+(d-2)\isph{u\,|\nabla_\omega\p|^2}\\
=\isph{u\,\|\mathrm L_\omega\p\|^2}+\frac1{d-1}\isph{u\,(\Delta_\omega\p)^2}+(d-2)\isph{u\,|\nabla_\omega\p|^2}
\end{multline*}
by~\eqref{TraceFree1}. Integrations by parts also show that
\[
\frac12\isph{u\,\Delta_\omega\(|\nabla_\omega\p|^2\)}=\frac12\isph{u\,|\nabla_\omega\p|^4}+\frac12\isph{u\,\Delta_\omega\p\,|\nabla_\omega\p|^2}\,,
\]
so that, by~\eqref{TraceFree2},
\begin{multline*}
\frac12\,\frac{(d-1)}{(d-2)}\isph{u\,\|\mathrm M_\omega\p\|^2}+\frac32\isph{u\,\Delta_\omega\p\,|\nabla_\omega\p|^2}+\frac{d-2}{d-1}\isph{u\,(\Delta_\omega\p)^2}\\
=\isph{u\,\|\mathrm L_\omega\p\|^2}+(d-2)\isph{u\,|\nabla_\omega\p|^2}\,.
\end{multline*}
Hence
\begin{multline}\label{Sphere:Second}
\isph{u\,(\Delta_\omega\p)^2}+\frac12\,\frac{(d-1)}{(d-2)}\(\frac{d-1}{d-2}\isph{u\,\|\mathrm M_\omega\p\|^2}+3\isph{u\,\Delta_\omega\p\,|\nabla_\omega\p|^2}\)\\
=\frac{d-1}{d-2}\isph{u\,\|\mathrm L_\omega\p\|^2}+(d-1)\isph{u\,|\nabla_\omega\p|^2}\,.
\end{multline}
We can now combine~\eqref{Sphere:First} and~\eqref{Sphere:Second} to get
\begin{multline*}
\isph{u\,(\Delta_\omega\p)^2}+\frac12\(\frac{d-1}{d-2}\)^2\isph{u\,\|\mathrm M_\omega\p\|^2}\\
-\frac 32\,\frac{(d-1)}{(d-2)}\,\frac{(d-1)}{(d+1)}\(\frac{d-1}{d-2}\isph{u\,\|\mathrm M_\omega\p\|^2}+2\isph{u\,\mathrm L_\omega\p:\mathrm M_\omega\p}\)\\
=\frac{d-1}{d-2}\isph{u\,\|\mathrm L_\omega\p\|^2}+(d-1)\isph{u\,|\nabla_\omega\p|^2}\,.
\end{multline*}
This allows us to prove that
\[
\isph{u\,\mathsf k[\p]}=\isph{u\(\mathsf a\,\|\mathrm L_\omega\p\|^2+\mathsf b\,\mathrm L_\omega\p:\mathrm M_\omega\p+\mathsf c\,\|\mathrm M_\omega\p\|^2\)}+(n-2)\(\alpha_{\rm FS}^2-\alpha^2\)\isph{u\,|\nabla_\omega\p|^2}
\]
where we use the fact that
\[
(d-2)-(n-2)\,\alpha^2+\(\frac1{d-1}-\frac1{n-1}\)(d-1)=(n-2)\(\alpha_{\rm FS}^2-\alpha^2\)
\]
with $\mathsf a=\frac{(d-1)\,(n-2)}{(d-2)\,(n-1)}$, $\mathsf b=\frac{3\,(d-1)\,(n-d)}{(n-1)\,(d+1)\,(d-2)}$ and $\mathsf c=\frac{(d-1)\,(n-d)}{(n-1)\,(d+1)\,(d-2)}$. Hence we obtain
\[
\isph{u\,\mathsf k[\p]}\ge(n-2)\(\alpha_{\rm FS}^2-\alpha^2\)\isph{u\,|\nabla_\omega\p|^2}+\(\mathsf c-\frac{\mathsf b^2}{4\,\mathsf a}\)\isph{u\,\|\mathrm M_\omega\p\|^2}
\]
because the discriminant $\mathsf b^2-4\,\mathsf a\,\mathsf c$ takes negative values. Taking into account~\eqref{TraceFree2}, this completes the proof with
\[
\delta=\(\mathsf c-\frac{\mathsf b^2}{4\,\mathsf a}\)\frac{d-2}{d-1}=(n-d)\,\frac{4\,(d+1)(d-2)+(4\,d-5)\,(n-d)}{4\,(n-1)\,(n-2)\,(d+1)^2}>0\,.
\]

If $d=2$, we identify $\S^1$ with $[0,2\pi)$, denote by $\theta\in[0,2\pi)$ the angular variable and by $u_\theta$ and $u_{\theta\theta}$ the first and second derivatives of $u$ with respect to $\theta$. As in~\cite[Lemma~5.3]{DEL-2015} and~\cite[Lemma~4.3]{Dolbeault2017}, we have
\[
\mathsf k[\p]=\frac{n-2}{n-1}\,|\p_{\theta\theta}|^2-(n-2)\,\alpha^2\,|\p_\theta|^2=(n-2)\,\(\alpha_{\rm FS}^2\,|\p_{\theta\theta}|^2-\,\alpha^2\,|\p_\theta|^2\)\,.
\]
Let $w=\sqrt u$ and recall that $u\,\p_\theta=u_\theta$ so that
\[
u\,|\p_\theta|^2=4\,|w_\theta|^2\quad\mbox{and}\quad u\,|\p_{\theta\theta}|^2=4\left|w_{\theta\theta}-\frac{|w_\theta|^2}w\right|^2\,.
\]
Notice that $w_{\theta\theta}\,|w_\theta|^2=\frac13\,\frac d{d\theta}(w_\theta\,|w_\theta|^2)$. With one integration by parts we obtain
\[
\icircle{\left|w_{\theta\theta}-\frac{|w_\theta|^2}w\right|^2}=\icircle{|w_{\theta\theta}|^2}+\icircle{\frac{|w_\theta|^4}{w^2}}-2\icircle{w_{\theta\theta}\,\frac{|w_\theta|^2}w}=\icircle{|w_{\theta\theta}|^2}+\frac13\icircle{\frac{|w_\theta|^4}{w^2}}\,.
\]
By the Poincar\'e inequality, we have
\[
\icircle{|w_{\theta\theta}|^2}\ge\icircle{|w_\theta|^2}
\]
and conclude that
\begin{multline*}
\icircle{u\,\mathsf k[\p]}\ge4\,(n-2)\,\(\alpha_{\rm FS}^2-\alpha^2\)\icircle{|w_\theta|^2}+\frac43\icircle{\frac{|w_\theta|^4}{w^2}}\\
=(n-2)\,\(\alpha_{\rm FS}^2-\alpha^2\)\icircle{u\,|p_\theta|^2}+\frac1{12}\icircle{u\,|p_\theta|^4}\,.
\end{multline*}
\end{proof}
\begin{lemma}\label{ExpDecay} \emph{With the above notations, we have the identity
\[
\iwrd{u\,x\,\cdot\D\(|\F|^2\)}-2\iwrd{u\,\F\cdot\D(\F\cdot x)}=-\,2\,\alpha\iwrd{u\,|\F|^2}\,.
\]
}\end{lemma}
\noindent Although very elementary, this estimate is fundamental as it establishes the exponential decay of the Fisher information in the symmetry range.  Lemma~\ref{ExpDecay} is in fact no more than an integration by parts.
\begin{proof} Since $x\cdot\D=\alpha\,r\,\partial_r$, $x\cdot\F=\alpha\,r\,\partial_r\,\mathsf p$, $x\cdot\nabla_\omega=0$, and $\F\cdot\partial_r\F=\F\cdot\partial_r(\D\mathsf p)=\F\cdot\D\,\partial_r\mathsf p-\frac1{r^3}\,|\nabla_\omega\mathsf p|^2$, we have that
\begin{multline*}
\iwrd{u\,x\,\cdot\D\(|\F|^2\)}-2\iwrd{u\,\F\cdot\D(\F\cdot x)}\\
=2\,\alpha\iwrd{u\(r\,\F\cdot\partial_r\F-r\,\F\cdot\D\partial_r\mathsf p-\alpha\,\partial_r\mathsf p\,(\omega\cdot\F)\)}\\
=-\,2\,\alpha\iwrd{u\(\alpha^2\,|\partial_r\mathsf p|^2+\frac{\left|\nabla_\omega\mathsf p\right|^2}{r^2}\)}=-\,2\,\alpha\iwrd{u\,|\F|^2}\,,
\end{multline*}
which concludes the proof.\end{proof}

\begin{proof}[Proof of Lemma~\ref{Lem:ExpDecay}]
Collecting the estimates of Lemmata~\ref{Lem:4.2} and~\ref{ExpDecay} into~\eqref{FK}, we have shown that
\[
\frac d{dt}\iwrd{u\,|\F|^2}+2\,\alpha\iwrd{u\,|\F|^2}\le-\,2\,(n-2)\(\alpha_{\rm FS}^2-\alpha^2\)\iwrd{u\,|\nabla_\omega\mathsf p|^2}-2\,\delta\iwrd{|\nabla_\omega\mathsf p|^4}\,.
\]
Under the assumptions of Lemma~\ref{Lem:ExpDecay}, the right-hand side is nonpositive, which completes the proof. \end{proof}

\begin{proof}[Proof of Theorem~\ref{Thm:WLSI1}] We learn from Lemma~\ref{Lem:ExpDecay} that
\[
\frac d{dt}\(\iwrd{u\,|\F|^2}-2\,\alpha\iwrd{u\,\p}\)\le0\,.
\]
On the other hand, $t\mapsto\iwrd{u(t,\cdot)\,|\F(t,\cdot)|^2}$ decays exponentially, whence $\lim_{t\to+\infty}\iwrd{u(t,\cdot)\,|\F(t,\cdot)|^2}=0$. Any decomposition of $u$ on a finite dimensional subspace of Hermite functions is exponentially decaying and such that $\lim_{t\to+\infty}\iwrd{u(t,\cdot)\,\p(t,\cdot)}=0$, thus proving that the inequality $\iwrd{u\,|\F|^2}\ge2\,\alpha\iwrd{u\,\p}$ is always true for any $t\ge0$ and, in particular, it holds true for the initial datum, which can be chosen arbitrarily. This amounts to~\eqref{WLSI} in the non scale-invariant form~\eqref{WLSI-alpha-unscaled}.
\end{proof}

\section{\texorpdfstring{\eqref{WLSI}}{WLSI} inequalities as an endpoint of some Caffarelli-Kohn-Nirenberg inequalities}\label{Sec:CKN}

This section relies on the results of~\cite{MR3579563,Dolbeault2017} and shows the consistency of our results with the symmetry properties of~\eqref{CKN} inequalities.

\subsection{A brief summary of the symmetry properties of some \texorpdfstring{\eqref{CKN}}{CKN} inequalities}\label{Sec:CKNsymmetry}

On the space $\mathrm H^{1,p}_{\beta,\gamma}(\R^d)$ of the functions $f\in\mathrm L^{p+1}_\gamma(\R^d)$, such that $\nabla f\in\mathrm L^2_\beta(\R^d)$, we consider the special family of \emph{Caffarelli-Kohn-Nirenberg interpolation inequalities}
\be{CKN}\tag{CKN}
\nrm f{2p,\gamma}\le\mathsf C_{\beta,\gamma,p}\,\nrm{\nabla f}{2,\beta}^\vartheta\,\nrm f{p+1,\gamma}^{1-\vartheta}\quad\forall\,f\in\mathrm H^{1,p}_{\beta,\gamma}(\R^d)
\ee with optimal constant $\mathsf C_{\beta,\gamma,p}$, and parameters $\beta$, $\gamma$ and $p$ such that
\be{parameters}
\gamma-2<\beta<\frac{d-2}d\,\gamma\,,\quad\gamma\in(-\infty,d)\,,\quad p\in\(1,p_\star\)\quad\mbox{with}\quad p_\star:=\frac{d-\gamma}{d-2-\beta}\,.
\ee
The exponent
\[
\vartheta=\frac{(d-\gamma)\,(p-1)}{p\,\big((d+2+\beta-2\,\gamma)-p\,(d-2-\beta)\big)}
\]
is determined by the invariance under scalings. The limitation $p\le p_\star$ in~\eqref{parameters} amounts, for a given $p>1$ to a restriction to the admissible set of parameters $(\beta,\gamma)$, namely
\be{LowerLine}
\beta\ge d-2-\frac{d-\gamma}p\,.
\ee
On the other hand, if $d\ge3$, we notice that the condition $p_\star<d/(d-2)$ is equivalent to
\[
\beta<\frac{d-2}d\,\gamma\,.
\]
The range of admissible parameters $(\beta,\gamma)$ is limited by~\eqref{parameters} to a cone in the quadrant $\beta<d-2$ and $\gamma<d$ with the additional condition~\eqref{LowerLine}. See~\cite[Fig.~1]{Dolbeault2017}.

The \emph{symmetry} versus \emph{symmetry breaking} issue is central in Caffarelli-Kohn-Nirenberg inequalities~\eqref{CKN}. Symmetry in~\eqref{CKN} means that the equality case is achieved by the (generalized) \emph{Aubin-Talenti type functions}
\be{Aubin-Talenti}
\mathsf g(x)=\big(1+|x|^{2+\beta-\gamma}\big)^{-\frac1{p-1}}\quad\forall\,x\in\R^d\,.
\ee
According to~\cite{Felli2003},~\cite[Theorem~2]{MR3579563} and in~\cite[Theorem~1.1]{Dolbeault2017} \emph{symmetry breaking} occurs~if and only if $(\beta,\gamma)$ satisfy~\eqref{eq:symmetrybreaking:range} where $\gamma\mapsto\beta_{\rm FS}(\gamma)$ is the \emph{Felli \& Schneider} curve defined by~\eqref{betaFS}. In the \emph{symmetry} range determined by~\eqref{eq:symmetry:range}, the value of $\mathsf C_{\beta,\gamma,p}$ is known. According to~\cite[Appendix~A]{MR3579563}, if~\eqref{eq:symmetry:range} holds, we have
\begin{equation}\label{eq:constant:rel}
\mathsf C_{\beta,\gamma,p}=\mathsf C_{\beta,\gamma,p}^\star
\end{equation}
where $\sigma_d:=|\S^{d-1}|=\frac{2\,\pi^{d/2}}{\Gamma(d/2)}$ is the volume of the unit sphere $\S^{d-1}\subset\R^d$,
\[
\mathsf C_{\beta,\gamma,p}^\star:=\alpha^\zeta\,\K
\]
where $n$ and $\alpha$ are given respectively by~\eqref{n} and~\eqref{alpha}, and
\begin{align*}
&\zeta:=\frac\vartheta2+\frac{1-\vartheta}{p+1}-\frac1{2\,p}=\frac{(2+\beta-\gamma)\,(p-1)}{2\,p\,\big(d+2+\beta-2\,\gamma\,-p\,(d-2-\beta)\big)}\,,\\
&\frac 1{\mathsf K_{\alpha,n,p}^\star}=\alpha^\vartheta\(\frac{4\,n}{p-1}\,\frac1{n+2-p\,(n-2)}\)^\frac\vartheta2\(\frac{2\,(p+1)}{n+2-p\,(n-2)}\)^\frac\vartheta{p+1}\(\sigma_d\,\frac{\Gamma(\frac n2)\,\Gamma(\frac{2\,p}{p-1}-\frac n2)}{2\,\Gamma(\frac{2\,p}{p-1})}\)^\zeta\,.
\end{align*}

\subsection{\texorpdfstring{\eqref{CKN}}{CKN} inequalities, the artificial dimension and the anisotropic gradient}\label{Sec:artificial.dimensionCKN}

Inequality~\eqref{CKN} can be recast as an interpolation inequality with the same weight in all integrals which, in terms of scaling properties, amounts to introduce an \emph{artificial dimension}. To a function $f\in\mathrm H^{1,p}_{\beta,\gamma}(\R^d)$, let us associate the function $F\in\mathrm H^{1,p}_{\nu,\nu}(\R^d)$ with $\nu:=d-n<0$ such that $f(x)=F\(|x|^{\alpha-1}\,x\)$ for any $x\in\R^d$ as in~\eqref{ChangeOfDimension}. Notice that $p_\star=n/(n-2)$. With $\alpha>0$ and $p\in(1,p_\star]$, we can rewrite~\eqref{CKN} as
\be{CKN1}
\nrm  F{2p,\nu}\le\mathsf K_{\alpha,n,p}\,\nrm{\,\DD  F}{2,\nu}^{\vartheta}\,\nrm  F{p+1,\nu}^{1-\vartheta}\quad\forall\, F\in \mathrm H^p_{\nu,\nu}(\R^d)\,,
\ee
for some optimal constant $\mathsf K_{\alpha,n,p}$ which is explicitly related to the optimal constant in~\eqref{CKN}: see~\cite[Proposition~6]{MR3579563}. Inequality~\eqref{CKN1} can be interpreted as a Gagliardo-Nirenberg-Sobolev inequality in the artificial dimension~$n$. As $\alpha\neq1$ unless $\beta=\gamma$, notice that symmetry issues in~\eqref{CKN1} are in no way simpler than in~\eqref{CKN}. A remarkable point is that the Aubin-Talenti type function as defined by~\eqref{Aubin-Talenti} is transformed, up to a scaling, into the more standard function
\[
x\mapsto\(1+\frac{p-1}2\,|x|^2\)^\frac1{1-p}\,,
\]
which converges to the standard gaussian function as $p\to1$. We refer to~\cite[Section~2.3]{MR3579563} and~\cite[Section~3.1]{DEL-2015} for further details. The limit of~\eqref{CKN1} as $p\to1_+$ is consistent with~\eqref{WLSI-alpha}. This is what we are going to exploit next.

\subsection{The limit as \texorpdfstring{$p\to1$}{pto1}}\label{Sec:pto1}

Assume that $\beta>\gamma-2$ so that~\eqref{LowerLine} is satisfied uniformly in the limit as $p\to1_+$. Inequality~\eqref{CKN} can be rewritten in logarithmic form as
\be{CKNlog}
\log\(\frac{\nrm f{2p,\gamma}}{\nrm f{p+1,\gamma}}\)\le\log\mathsf C_{\beta,\gamma,p}+\vartheta(p)\,\log\(\frac{\nrm{\nabla f}{2,\beta}}{\nrm f{p+1,\gamma}}\)\,.
\ee
It is clear from~\eqref{CKN} that $\lim_{p\to1_+}\mathsf C_{\beta,\gamma,p}=1$ and 
both sides in~\eqref{CKNlog} vanish in the limit as $p\to1_+$, so that the inequality degenerates into an equality. Let us divide both sides of~\eqref{CKNlog} by $(p-1)$ and consider the limit. Using the identity
\[
\frac d{dq}\log\nrm f{q,\gamma}=\frac1{q^2}\ird{\frac{|f|^q}{\nrm f{q,\gamma}^q}\,\log\(\frac{|f|^q}{\nrm f{q,\gamma}^q}\)\,|x|^{-\gamma}}\,,
\]
$\lim_{p\to1_+}\vartheta(p)/(p-1)=n/4$ where $n$ is given by~\eqref{n} and
\[
\limsup_{p\to1_+}\frac1{p-1}\,\log\(\frac{\nrm f{2p,\gamma}}{\nrm f{p+1,\gamma}}\)=\frac14\ird{\frac{|f|^2}{\nrm f{2,\gamma}^2}\,\log\(\frac{|f|^2}{\nrm f{2,\gamma}^2}\)\,|x|^{-\gamma}}\,,
\]
we can pass to the limit as $p\to1_+$. The overall picture is consistent with logarithmic Sobolev inequalities~\eqref{WLSI}. In the limit as $p\to1_+$, it is straightforward to see that the conditions that define the symmetry range~\eqref{eq:symmetry:range} in~\eqref{CKN} provide us with the conditions that define the symmetry range in~\eqref{WLSI} as stated in Theorem~\ref{Thm:WLSI1}. This is also true at the level of the optimal constants in the symmetry range. In fact, these observations provide us with an alternative strategy of proof of Theorem~\ref{Thm:WLSI1} based on~\cite[Theorem~1.1]{Dolbeault2017} using $\Gamma$-convergence methods in the spirit of~\cite{DELT09,Dolbeault2016}. We do not expand on this as we already have a direct proof but for consistency, we state the following result. 
\begin{proposition}\label{Prop:CKNlimit} \emph{Let $d\ge2$ and assume that $(\beta,\gamma)\neq(0,0)$ satisfies~\eqref{Range}. Then we have
\[
\mathscr C_{\beta,\gamma}\le\mathscr C_{\beta,\gamma}^\star:=4\,\limsup_{p\to1_+}\frac{\mathsf C_{\beta,\gamma,p}^\star-1}{p-1}\,.
\]
}\end{proposition}
\begin{proof}
In view of~\eqref{eq:constant:rel}, in the symmetry range for the parameters~\eqref{eq:symmetry:range}, we can directly differentiate the formula
\begin{equation}\label{eq:log:const}
\log\,\mathsf C_{\beta,\gamma,p}^\star=\zeta(p)\,\log\,\alpha + \log\K
\end{equation}
where $\vartheta=n\,\zeta(p)$ and
\[
\frac1{\K}=\alpha^{n\,\zeta(p)}
\(\frac{4\,n}{b(p)}\,\frac1{p-1}\)^{\frac n2\,\zeta(p)}
\(\frac{2\,(p+1)}{b(p)}\)^{\frac n{p+1}\,\zeta(p)}
\(\frac12\,\sigma_d\,\Gamma\big(\tfrac n2\big)\,\frac{\Gamma\big(\frac{2\,p}{p-1}-\frac n2\big)}{\Gamma\big(\frac{2\,p}{p-1}\big)}\)^{\zeta(p)}
\]
and $b(p):=n+2-p\,(n-2)$. In particular, note that $b(1)=4$ and $b'(1)=-\,(n-2)$. By taking the logarithm, we obtain the identity
\begin{equation}\label{eq:decompos:opt:const}
-\,\log\K
= \zeta(p)
\(n\log \alpha + \frac n{p+1}\log\(\frac{2\,(p+1)}{b(p)}\)+\log\(\frac12\,\sigma_d\,\Gamma\big(\tfrac n2\big)\)+f(p)\)
\end{equation}
where
\[
f(p):=\frac n2\,\log\(\frac{4\,n}{b(p)}\,\frac1{p-1}\)
+\log\(\frac{\Gamma\big(\frac{2\,p}{p-1}-\frac n2\big)}{\Gamma\big(\frac{2\,p}{p-1}\big)}\).
\]
By using the asymptotic expansion for the \emph{Gamma} function
\[
\lim_{x\to+\infty}\frac{\Gamma(x+\alpha)}{x^\alpha\,\Gamma(x)}=1
\]
for any $\alpha\in\R$, one can compute the limits
\[
\lim_{p\to 1_+}f(p)=-\,\frac 32\,n\,\log2\,.
\]
As a result, we can take the derivative with respect to $p$ and evaluate the limit as $p\to 1_+$ in~\eqref{eq:decompos:opt:const} by
\[
-\,\lim_{p\to 1_+}\frac{\log\K}{p-1}=\frac14\,\log\(\frac12\,\sigma_d\,\alpha^n\,\(\tfrac{n\,e}2\)^\frac n2\Gamma\(\tfrac n2\)\)
\]
using the fact that $\zeta(p)=\alpha\,\frac{p-1}{p\,b(p)}$ is such that $\zeta(1)=\lim_{p\to1_+}\zeta(p)=0$ and $\zeta'(1)=1/4$. With $n$ and $\alpha$ given in terms of $\beta$ and $\gamma$ respectively by~\eqref{n} and~\eqref{alpha}, we deduce from~\eqref{eq:log:const} that $\mathscr C_{\beta,\gamma}^\star=-\,\log\big(\frac12\,\sigma_d\,\alpha^{n-1}\,\(\tfrac{n\,e}2\)^{n/2}\,\Gamma\big(\tfrac n2\big)\big)$.
\end{proof}

\section{Some consequences for weighted diffusion flows}\label{Sec:Flows}

\subsection{Self-similar solutions, intermediate asymptotics and entropy decay rates}\label{Sec:Self-similarSolutions}

Let us consider the self-similar change of variables
\[
u(t,x)=R(t)^{\gamma-d}\,v\(\log R(t),\frac x{R(t)}\)\quad\mbox{with}\quad\frac{dR}{dt}=R^{\gamma-\beta-1}=R^{1-2\,\alpha}
\]
which transforms~\eqref{heat} into the \emph{weighted Fokker-Planck equation}
\be{FP}
\frac{\partial v}{\partial t}=|x|^\gamma\,\nabla\cdot\(|x|^{-\beta}\,\nabla v+x\,|x|^{-\gamma}\,v\)\,.
\ee
A simple stationary solution is given by $v_\star(x)=c_{n,d}\,(2/\alpha)^{n/4}\,\exp\big(-\frac1{2\,\alpha}\,|x|^{2\,\alpha}\big)=g_\star^{\alpha,\alpha}(x)$ with the notations of Section~\ref{Sec:Intro}. With $R(0)=R_0\ge0$, we find that
\[
R(t)=\(R_0^{2\alpha}+2\,\alpha\,t\)^\frac1{2\alpha}\quad\forall\,t\ge0\,,
\]
which shows that $u_\star(t,x)=R(t)^{\gamma-d}\,f_\star\big(\log R(t),x/R(t)\big)$ is simply the Green function associated to~\eqref{heat} if we choose $R_0=0$. Another interesting choice of $R_0$ is $R_0=1$ so that the initial datum for~\eqref{FP} is the same as for~\eqref{heat}. If $v$ solves~\eqref{FP}, then the function $w=v/v_\star$ solves the \emph{weighted Ornstein-Uhlenbeck equation}
\be{OU}
\frac{\partial w}{\partial t}=\frac{|x|^\gamma}{v_\star(x)}\,\nabla\cdot\(|x|^{-\beta}\,v_\star\,\nabla w\)\,.
\ee
\begin{proposition}\label{Prop:EEP}\emph{In the symmetry range, with $d\mu_\alpha=|x|\,^{-\gamma}\,v_\star(x)\,dx$ defined as in Section~\ref{Sec:Intro}, any solution of~\eqref{OU} with nonnegative initial datum $w_0$ such that $\int_{\R^d}w_0\,d\mu_\alpha=1$ decays according to
\[
\int_{\R^d}w(t,\cdot)\,\log w(t,\cdot)\,d\mu_\alpha\le \(\int_{\R^d}w_0\,\log w_0\,d\mu_\alpha\)e^{-\,4\,\alpha\,t}\quad\forall\,t\ge0\,.
\]
}\end{proposition}
\begin{proof} We compute $\frac d{dt}\int_{\R^d}w(t,\cdot)\,\log w(t,\cdot)\,d\mu_\alpha$ and apply~\eqref{Gaussian-1} to $u=\sqrt w$. \end{proof}

By the Csisz\'ar-Kullback-Pinsker inequality
\[
\int_{\R^d}|w-1|\,d\mu_\alpha\le\frac12\,\sqrt{\textstyle{\int_{\R^d}w\,\log w\,d\mu_\alpha}}
\]
for any nonnegative function $w$ such that $\int_{\R^d}w\,d\mu_\alpha=1$. By undoing the above changes of variables with $R_0=1$, we can write the following \emph{intermediate asymptotics} result.
\begin{corollary}\label{Cor:IA}\emph{In the symmetry range, any solution of~\eqref{heat} with nonnegative initial datum $u_0$ such that $\nrm{u_0}{1,\gamma}=1$ obeys to
\[
\nrm{u(t,\cdot)-u_\star(t,\cdot)}{1,\gamma}\le\frac12\,\sqrt{\textstyle{\ird{u_0\,\log(u_0/v_\star)\,|x|^{-\gamma}}}}\,\(1+2\,\alpha\,t\)^{-1}\quad\forall\,t\ge0\,.
\]
}\end{corollary}

The above results are consistent with the flow
\[
\frac{\partial u}{\partial t}=|x|^\gamma\,\nabla\cdot\(|x|^{-\beta}\,\nabla u^m\)
\]
with $m<1$, which is adapted to~\eqref{CKN} inequalities with $p=1/(2\,m-1)$. Notice however that parabolic computations as in Section~\ref{Sec:CarreDuChamp} are, so far, only formal if $m<1$: see~\cite{MR3579563,DEL-2015,Dolbeault2017,DEL-JEPE} for details.

\subsection{Hyper-contractivity estimates}

Let us measure the gain of regularity by the weighted heat flow~\eqref{heat}. the following result generalizes~\cite{MR0343816,Federbush,Gross75}.
\begin{proposition}\label{Prop:Nelson} \emph{Let $d\ge1$, $r>q>1$ and assume that $\beta$ and $\gamma$ satisfy~\eqref{Range}. If $u$ is a solution of~\eqref{heat} with initial datum $u_0\in\mathrm L^q_\gamma(\R^d)$, then
\be{Ineq:Nelson}
\nrm{u(t,\cdot)}{r,\gamma}\le\mathscr H^{\,q,r}_{\beta,\gamma}\,\nrm{u_0}{q,\gamma}\,t^{-\frac n2\,\frac{r-q}{q\,r}}\quad\forall\,t\ge0
\ee
where $\mathscr H^{\,q,r}_{\beta,\gamma}:=t_\star^{\frac n2\,\frac{r-q}{q\,r}}$ and $t_\star:=\tfrac n8\,e^{\frac2n\,\mathscr{C}_{\beta,\gamma}-1}\log\big(\tfrac{r-1}{q-1}\big)$.
}\end{proposition}
\begin{proof} For some exponent $p$ depending smoothly on $s$ with $p'(s)>0$, let us consider the function
\[
h(s):=\nrm{u(s,\cdot)}{p(s),\gamma}\,.
\]
By a standard computation which goes back to~\cite{Gross75} we have
\be{EDOp}
\frac{h'}h=\frac{p'}{p^2}\ird{\frac{|u|^p}{h^p}\,\log\(\frac{|u|^p}{h^p}\)|x|^{-\gamma}}-\frac1{h^p}\,\frac{4\,(p-1)}{p^2}\ird{\left|\nabla|u|^{p/2}\right|^2\,|x|^{-\beta}}\le\frac{p'}{p^2}\(\mathscr C_{\beta,\gamma}-\frac n2\,\log\(\frac{2\,e}{n\,\sigma}\)\)
\ee
where the inequality holds as a consequence of~\eqref{WLSI-unscaled} applied to $|u|^{p/2}$ with
\be{EDOp2}
p'=4\,\sigma\,(p-1)\,.
\ee
With the choice $\sigma=\sigma_{\beta,\gamma}$ where $\sigma_{\beta,\gamma}:=\frac2n\,e^{1-\frac2n\,\mathscr{C}_{\beta,\gamma}}$, so that $h'\le0$, and $p(0)=q$,~\eqref{EDOp} is solved by
\be{SolnEDO}
p(s)=1+(q-1)\,e^{4\,\sigma_{\beta,\gamma}\,s}\quad\forall\,s\ge0\,.
\ee
The condition $p(t_\star)=r$ determines
\[
t_\star=\frac1{4\,\sigma_{\beta,\gamma}}\,\log\(\frac{r-1}{q-1}\)
\]
such that
\[
\nrm{u(t_\star,\cdot)}{r,\gamma}=\nrm{u_0}{q,\gamma}\,.
\]
If $t\neq t_\star$, we use~\eqref{EDOp} again for $\sigma>0$ such that $r=1+(q-1)\,e^{4\,\sigma\,t}$, \emph{i.e.},
\be{Choice:t}
t=\frac1{4\,\sigma}\,\log\(\frac{r-1}{q-1}\)\,,
\ee
and obtain
\[
\frac{h'}h\le\frac{p'}{p^2}\,\frac n2\,\log\(\frac{\sigma}{\sigma_{\beta,\gamma}}\)\,,
\]
that is, after integration with respect to $s\in[0,t]$,
\[
\nrm{u(t,\cdot)}{r,\gamma}=\nrm{u_0}{q,\gamma}\,\big(t_\star^{-1}\,t\big)^{-\frac n2\,\frac{r-q}{r\,q}}\quad\forall\,t\ge0\,.
\]

\end{proof}

Notice that the choice of $t_\star$ in~\eqref{Ineq:Nelson} is optimal because~\eqref{Ineq:Nelson} with $\sigma=\sigma_{\beta,\gamma}$ means that $h(s)\le h(0)$ for any $s>0$, hence $h'(0)\le0$ so that the optimal value of $\sigma_{\beta,\gamma}$ in~\eqref{SolnEDO} determines the optimal constant in~\eqref{WLSI-unscaled}. Slightly more subtle is the fact that $\mathscr H^{\,q,r}_{\beta,\gamma}$ is also the optimal constant. Using~\eqref{EDOp} with the condition $p(t)=r$, we can write that
\[
\nrm{u(t,\cdot)}{r,\gamma}=h(t)\le h(0)\,\exp\(\int_0^t\frac{p'(s)}{p^2(s)}\(\mathscr C_{\beta,\gamma}-\frac n2\,\log\(\frac{2\,e}{n\,\sigma}\)\)ds\)
\]
where $h(0)=\nrm{u_0}{q,\gamma}$ and $\sigma$ can be taken $s$-dependent. With the change of variables $s\mapsto z$, $z=Z(s)=1/p(s)$, we can compute
\[
\int_0^t\frac{p'(s)}{p^2(s)}\(\mathscr C_{\beta,\gamma}-\frac n2\,\log\(\frac{2\,e}{n\,\sigma}\)\)ds=\(\mathscr C_{\beta,\gamma}-\frac n2\,\log\(\frac{2\,e}n\)\)\(\frac1q-\frac1r\)-\frac n2\int_{1/r}^{1/q}\log\sigma(z)\,dz
\]
where, up to a slight abuse of notations, we consider $\sigma$ as a function of $z$ and deduce from~\eqref{EDOp2} that
\[
\sigma(z)=4\,\frac{p^2}{p'}\,\frac{p-1}{p^2}=4\,\frac{z\,(z-1)}{(Z'\circ Z^{-1})(z)}\,,
\]
although we do not make use of this identity. Indeed, an infinitesimal variation of $\int_{1/r}^{1/q}\log\sigma(z)\,dz$ directly shows that the optimal case is achieved by a constant function $z\mapsto\sigma(z)$ corresponding to the choice~\eqref{Choice:t}. Hence $\mathscr H^{\,q,r}_{\beta,\gamma}$ as defined in Proposition~\ref{Prop:Nelson} is optimal.

\bigskip\begin{center}\rule{2cm}{0.5pt}\end{center}\bigskip\appendix
\renewcommand{\thesection}{{\bf Appendix~\Alph{section}.}}
\section{A purely algebraic computation}\label{Appendix}
\renewcommand{\thesection}{\Alph{section}}

For completeness, let us give a proof of Lemma~\ref{Lem:4.2}, (i). We recall that
\[
\mathsf K[\D\mathsf p]=\frac12\,\mathsf L_\alpha\,|\DD\p|^2-\,\DD\p\cdot\DD\mathsf L_\alpha\p-\frac1n\,(\mathsf L_\alpha\p)^2
\]
and
\[
\mathsf k[\p]=\frac12\,\Delta_\omega\,|\nabla_\omega\p|^2-\nabla_\omega\p\cdot\nabla_\omega\Delta_\omega\,\p-\frac1{n-1}\,(\Delta_\omega\,\p)^2-(n-2)\,\alpha^2\,|\nabla_\omega\p|^2\,.
\]
\par\smallskip\begin{lemma}\label{Lem:Derivmatrixform1}{\sl Let $d\in\N$, $n\in\R$ such that $n>d\ge2$, and consider a function $\p\in C^3(\R^d\setminus\{0\})$. Then,
\[
\mathsf K[\D\mathsf p]=\alpha^4\(1-\frac1n\)\left|\p''-\frac{\p'}\rs-\frac{\Delta_\omega\,\p}{\alpha^2\,(n-1)\,\rs^2}\right|^2+\frac{2\,\alpha^2}{\rs^2}\left|\nabla_\omega\p'-\frac{\nabla_\omega\p}\rs \right|^2+\frac{\mathsf k[\p]}{\rs^4}\,.
\]}\end{lemma}\par\smallskip
This result is a purely algebraic computation which involves no integration by parts and in which, for~\eqref{CKN}, neither $p$ nor $m$ plays any role, so that it perfectly makes sense to consider the limit case $m=1$ and $p=1$. We recall that the original result of~\cite[Lemma~5.1]{DEL-2015} was given in the framework of $m<1$.
\begin{proof} By definition of $\mathsf K[\D\mathsf p]$, we have
\begin{eqnarray*}
\mathsf K[\D\mathsf p]&=&\frac{\alpha^2}2\left[\alpha^2\,\p'^2+\frac{|\nabla_\omega\p|^2}{\rs^2}\)''+\frac{\alpha^2}2\frac{(n-1)}\rs
\left[\alpha^2\,\p'^2+\frac{|\nabla_\omega\p|^2}{\rs^2}\)'+\frac1{2\,\rs^2}\,\Delta_\omega\left[\alpha^2\,\p'^2+\frac{|\nabla_\omega\p|^2}{\rs^2}\)\\
&&-\,\alpha^2\,\p'\(\alpha^2\,\p''+\alpha^2\,\frac{(n-1)}\rs\,\p'+\frac{\Delta_\omega\,\p}{\rs^2}\)'-\frac1{\rs^2}
\nabla_\omega\p\cdot\nabla_\omega\(\alpha^2\,\p''+\alpha^2\,\frac{(n-1)}\rs\,\p'+\frac{\Delta_\omega\,\p}{\rs^2}\)\\
&&-\,\frac 1n\(\alpha^2\,\p''+\alpha^2\,\frac{(n-1)}\rs\,\p'+\frac{\Delta_\omega\,\p}{\rs^2}\)^2\,,
\end{eqnarray*}
which can be expanded as
\begin{eqnarray*}
\mathsf K[\D\mathsf p]\kern-6pt&=&\kern-6pt\frac{\alpha^2}2\left[ 2\,\alpha^2\,\p''^2+2\,\alpha^2\,\p'\,\p'''+2\,\frac{|\nabla_\omega\p'|^2+\nabla_\omega\p\cdot\nabla_\omega\p''}{\rs^2}
-8\,\frac{\nabla_\omega\p\cdot\nabla_\omega\p'}{\rs^3}+6\,\frac{|\nabla_\omega\p|^2}{\rs^4}\)\\
&&\kern-6pt+\,\alpha^2\,\frac{(n-1)}\rs\left[\alpha^2\,\p'\,\p''+\frac{\nabla_\omega\p\cdot\nabla_\omega\p'}{\rs^2}-\frac{|\nabla_\omega\p|^2}{\rs^3}\)+\frac1{\rs^2}\left[\alpha^2\,\p'\Delta_\omega\,\p'+\alpha^2\,|\nabla_\omega\p'|^2+\frac{\Delta_\omega\,|\nabla_\omega\p|^2}{2\,\rs^2}\)\\
&\kern-6pt&-\,\alpha^2\,\p'\(\alpha^2\,\p'''+\alpha^2\,\frac{(n-1)}\rs\,\p''-\,\alpha^2\,\frac{(n-1)}{\rs^2}\p'-2\,\frac{\Delta_\omega\,\p}{\rs^3}+\frac{\Delta_\omega\,\p'}{\rs^2}\)\\
&&\kern-6pt\hspace*{2cm}-\frac1{\rs^2}
\(\alpha^2\,\nabla_\omega\p\cdot\nabla_\omega\p''+\alpha^2\,\frac{(n-1)}\rs\nabla_\omega\p\cdot\nabla_\omega\p'+\frac{\nabla_\omega\p\cdot\nabla_\omega\Delta_\omega\,\p}{\rs^2}\)\\
&&\kern-6pt-\,\frac 1n\left[\alpha^4\,\p''^2+\alpha^4\,\frac{(n-1)^2}{\rs^2}\,\p'^2+\frac{(\Delta_\omega\,\p)^2}{\rs^4}+2\,\alpha^4\,\frac{(n-1)}\rs\,\p'\,\p''+2\,\alpha^2\,\frac{\p''\Delta_\omega\,\p}{\rs^2}+2\,\alpha^2\,\frac{(n-1)}{\rs^3}\p'\Delta_\omega\,\p\).
\end{eqnarray*}
Collecting terms proves the result.\end{proof}

\bigskip\noindent{\bf Acknowledgments.} This work has been supported by the Project EFI (ANR-17-CE40-0030) of the French National Research Agency (ANR). AZ was funded by ANID Chile under the grant FONDECYT de Iniciaci\'on en Investigaci\'on $N^{\circ}$ 11201259, and supported by Instituto de Ciencias de la Ingenier\'ia (ICI) of Universidad de O'Higgins (UOH) under the fund Fondo de Instalaci\'on y Movilidad.


\small

\begin{thebibliography}{10}

\bibitem{MR4372142}
{\sc R.~Adamczak, B.~Polaczyk, and M.~Strzelecki}, {\em Modified log-{S}obolev
  inequalities, {B}eckner inequalities and moment estimates}, J. Funct. Anal.,
  282 (2022), pp.~Paper No. 109349, 76.

\bibitem{MR1845806}
{\sc C.~An\'{e}, S.~Blach\`ere, D.~Chafa\"{\i}, P.~Foug\`eres, I.~Gentil,
  F.~Malrieu, C.~Roberto, and G.~Scheffer}, {\em Sur les in\'{e}galit\'{e}s de
  {S}obolev logarithmiques}, vol.~10 of Panoramas et Synth\`eses [Panoramas and
  Syntheses], Soci\'{e}t\'{e} Math\'{e}matique de France, Paris, 2000.
\newblock With a preface by Dominique Bakry and Michel Ledoux.

\bibitem{Bakry1985}
{\sc D.~Bakry and M.~\'{E}mery}, {\em Diffusions hypercontractives}, in
  S\'{e}minaire de probabilit\'{e}s, {XIX}, 1983/84, vol.~1123 of Lecture Notes
  in Math., Springer, Berlin, 1985, pp.~177--206.

\bibitem{MR3155209}
{\sc D.~Bakry, I.~Gentil, and M.~Ledoux}, {\em Analysis and geometry of
  {M}arkov diffusion operators}, vol.~348 of Grundlehren der Mathematischen
  Wissenschaften [Fundamental Principles of Mathematical Sciences], Springer,
  Cham, 2014.

\bibitem{MR2320410}
{\sc F.~Barthe, P.~Cattiaux, and C.~Roberto}, {\em Interpolated inequalities
  between exponential and {G}aussian, {O}rlicz hypercontractivity and
  isoperimetry}, Rev. Mat. Iberoam., 22 (2006), pp.~993--1067.

\bibitem{BV-V}
{\sc M.-F. Bidaut-V{\'e}ron and L.~V{\'e}ron}, {\em Nonlinear elliptic
  equations on compact {R}iemannian manifolds and asymptotics of {E}mden
  equations}, Invent. Math., 106 (1991), pp.~489--539.

\bibitem{Bobkov_2009}
{\sc S.~G. Bobkov and M.~Ledoux}, {\em Weighted {P}oincar{\'{e}}-type
  inequalities for {C}auchy and other convex measures}, Ann. Probab., 37
  (2009).

\bibitem{MR3579563}
{\sc M.~Bonforte, J.~Dolbeault, M.~Muratori, and B.~Nazaret}, {\em Weighted
  fast diffusion equations ({P}art~{I}): {S}harp asymptotic rates without
  symmetry and symmetry breaking in {C}affarelli-{K}ohn-{N}irenberg
  inequalities}, Kinet. Relat. Models, 10 (2017), pp.~33--59.

\bibitem{Bonforte201761}
\leavevmode\vrule height 2pt depth -1.6pt width 23pt, {\em Weighted fast
  diffusion equations ({P}art~{II}): Sharp asymptotic rates of convergence in
  relative error by entropy methods}, Kinet. Relat. Models, 10 (2017),
  pp.~61--91.

\bibitem{BDNS2021}
{\sc M.~Bonforte, J.~Dolbeault, B.~Nazaret, and N.~Simonov}, {\em Stability in
  {G}a\-gliar\-do-{N}irenberg-{S}obolev inequalities: flows, regularity and the
  entropy method}.
\newblock Preprint
  \href{https://hal.archives-ouvertes.fr/hal-02887010}{hal-02887010} and
  \href{https://arxiv.org/abs/2007.03674}{arXiv: 2007.03674}, to appear in
  \emph{Memoirs of the AMS}.

\bibitem{bonforte:hal-03581542}
\leavevmode\vrule height 2pt depth -1.6pt width 23pt, {\em {Constructive
  stability results in interpolation inequalities and explicit improvements of
  decay rates of fast diffusion equations}}.
\newblock Preprint
  \href{https://hal.archives-ouvertes.fr/hal-03581542}{hal-03581542} and
  \href{https://arxiv.org/abs/2202.09693}{arXiv: 2202.09693}, Feb. 2022.

\bibitem{zbMATH03834677}
{\sc H.~{Br\'ezis} and E.~H. {Lieb}}, {\em {A relation between pointwise
  convergence of functions and convergence of functionals}}, {Proc. Am. Math.
  Soc.}, 88 (1983), pp.~486--490.

\bibitem{BDS}
{\sc G.~Brigati, J.~Dolbeault, and N.~Simonov}, {\em Stability for the
  logarithmic {S}obolev inequality}.
\newblock Work in progress, 2022.

\bibitem{Caffarelli1984}
{\sc L.~Caffarelli, R.~Kohn, and L.~Nirenberg}, {\em First order interpolation
  inequalities with weights}, Compos. Math., 53 (1984), pp.~259--275.

\bibitem{MR1132315}
{\sc E.~A. Carlen}, {\em Superadditivity of {F}isher's information and
  logarithmic {S}obolev inequalities}, J. Funct. Anal., 101 (1991),
  pp.~194--211.

\bibitem{Carrillo2001}
{\sc J.~A. Carrillo, A.~J\"{u}ngel, P.~A. Markowich, G.~Toscani, and
  A.~Unterreiter}, {\em Entropy dissipation methods for degenerate parabolic
  problems and generalized {S}obolev inequalities}, Monatsh. Math., 133 (2001),
  pp.~1--82.

\bibitem{Catrina2001}
{\sc F.~Catrina and Z.-Q. Wang}, {\em On the {C}affarelli-{K}ohn-{N}irenberg
  inequalities: sharp constants, existence (and nonexistence), and symmetry of
  extremal functions}, Comm. Pure Appl. Math., 54 (2001), pp.~229--258.

\bibitem{MR2609591}
{\sc P.~Cattiaux, N.~Gozlan, A.~Guillin, and C.~Roberto}, {\em Functional
  inequalities for heavy tailed distributions and application to isoperimetry},
  Electron. J. Probab., 15 (2010), pp.~no. 13, 346--385.

\bibitem{MR3603301}
{\sc P.~Cattiaux and A.~Guillin}, {\em Hitting times, functional inequalities,
  {L}yapunov conditions and uniform ergodicity}, J. Funct. Anal., 272 (2017),
  pp.~2361--2391.

\bibitem{MR4013832}
{\sc P.~Cattiaux, A.~Guillin, P.~Monmarch\'{e}, and C.~Zhang}, {\em Entropic
  multipliers method for {L}angevin diffusion and weighted log {S}obolev
  inequalities}, J. Funct. Anal., 277 (2019), pp.~108288, 24.

\bibitem{MR3008255}
{\sc P.~Cattiaux, A.~Guillin, and L.-M. Wu}, {\em Some remarks on weighted
  logarithmic {S}obolev inequality}, Indiana Univ. Math. J., 60 (2011),
  pp.~1885--1904.

\bibitem{chatzakou2021logarithmic}
{\sc M.~Chatzakou, A.~Kassymov, and M.~Ruzhansky}, {\em Logarithmic
  {H}ardy-{R}ellich inequalities on {L}ie groups}.
\newblock Preprint \href{https://arxiv.org/abs/2107.04874}{arXiv: 2107.04874},
  2021.

\bibitem{MR1132797}
{\sc C.~C. Chen and C.~S. Lin}, {\em Uniqueness of the ground state solutions
  of {$\Delta u+f(u)=0$} in {$\mathbb R^n,\;n\geq 3$}}, Comm. Partial
  Differential Equations, 16 (1991), pp.~1549--1572.

\bibitem{MR3647065}
{\sc X.~Chen and J.~Wang}, {\em Weighted {P}oincar\'{e} inequalities for
  non-local {D}irichlet forms}, J. Theoret. Probab., 30 (2017), pp.~452--489.

\bibitem{Chou-Chu-93}
{\sc K.~S. Chou and C.~W. Chu}, {\em On the best constant for a weighted
  {S}obolev-{H}ardy inequality}, J. London Math. Soc. (2), 48 (1993),
  pp.~137--151.

\bibitem{MR823597}
{\sc M.~H.~M. Costa}, {\em A new entropy power inequality}, IEEE Trans. Inform.
  Theory, 31 (1985), pp.~751--760.

\bibitem{MR4186674}
{\sc U.~Das}, {\em On weighted logarithmic-{S}obolev \& logarithmic-{H}ardy
  inequalities}, J. Math. Anal. Appl., 496 (2021), pp.~Paper No. 124796, 30.

\bibitem{delPino20102045}
{\sc M.~del Pino, J.~Dolbeault, S.~Filippas, and A.~Tertikas}, {\em A
  logarithmic {H}ardy inequality}, J. Funct. Anal., 259 (2010), pp.~2045 --
  2072.

\bibitem{MR2381156}
{\sc J.~Demange}, {\em Improved {G}agliardo-{N}irenberg-{S}obolev inequalities
  on manifolds with positive curvature}, J. Funct. Anal., 254 (2008),
  pp.~593--611.

\bibitem{DE2010}
{\sc J.~Dolbeault and M.~J. Esteban}, {\em Extremal functions for
  {C}affarelli-{K}ohn-{N}irenberg and logarithmic {H}ardy inequalities}, Proc.
  Roy. Soc. Edinburgh Sect. A, 142 (2012), pp.~745--767.

\bibitem{DoEsFiTe2014}
{\sc J.~Dolbeault, M.~J. Esteban, S.~Filippas, and A.~Tertikas}, {\em Rigidity
  results with applications to best constants and symmetry of
  {C}affarelli-{K}ohn-{N}irenberg and logarithmic {H}ardy inequalities}, Calc.
  Var. Partial Differential Equations, 54 (2015), pp.~2465--2481.

\bibitem{DEKL2014}
{\sc J.~Dolbeault, M.~J. Esteban, M.~Kowalczyk, and M.~Loss}, {\em Improved
  interpolation inequalities on the sphere}, Discrete Contin. Dyn. Syst. Ser.
  S, 7 (2014), pp.~695--724.

\bibitem{Dolbeault20141338}
{\sc J.~Dolbeault, M.~J. Esteban, and M.~Loss}, {\em Nonlinear flows and
  rigidity results on compact manifolds}, J. Funct. Anal., 267 (2014), pp.~1338
  -- 1363.

\bibitem{DEL-JEPE}
\leavevmode\vrule height 2pt depth -1.6pt width 23pt, {\em Interpolation
  inequalities, nonlinear flows, boundary terms, optimality and linearization},
  J. Elliptic Parabol. Equ., 2 (2016), pp.~267--295.

\bibitem{DEL-2015}
\leavevmode\vrule height 2pt depth -1.6pt width 23pt, {\em Rigidity versus
  symmetry breaking via nonlinear flows on cylinders and {E}uclidean spaces},
  Invent. Math., 206 (2016), pp.~397--440.

\bibitem{1504}
\leavevmode\vrule height 2pt depth -1.6pt width 23pt, {\em Interpolation
  inequalities on the sphere: linear vs. nonlinear flows}, Ann. Fac. Sci.
  Toulouse Math. (6), 26 (2017), pp.~351--379.

\bibitem{1703}
\leavevmode\vrule height 2pt depth -1.6pt width 23pt, {\em Symmetry and
  symmetry breaking: rigidity and flows in elliptic {PDE}s.}, Proc. Int. Cong.
  of Math. 2018, Rio de Janeiro, 3 (2018), pp.~2279--2304.

\bibitem{Dolbeault2017}
{\sc J.~Dolbeault, M.~J. Esteban, M.~Loss, and M.~Muratori}, {\em Symmetry for
  extremal functions in subcritical {C}affarelli-{K}ohn-{N}irenberg
  inequalities}, C. R. Math. Acad. Sci. Paris, 355 (2017), pp.~133--154.

\bibitem{DELT09}
{\sc J.~Dolbeault, M.~J. Esteban, M.~Loss, and G.~Tarantello}, {\em On the
  symmetry of extremals for the {C}affarelli-{K}ohn-{N}irenberg inequalities},
  Adv. Nonlinear Stud., 9 (2009), pp.~713--726.

\bibitem{MR2366398}
{\sc J.~Dolbeault, I.~Gentil, A.~Guillin, and F.-Y. Wang}, {\em
  {$L^q$}-functional inequalities and weighted porous media equations},
  Potential Anal., 28 (2008), pp.~35--59.

\bibitem{Dolbeault2016}
{\sc J.~Dolbeault, M.~Muratori, and B.~Nazaret}, {\em Weighted interpolation
  inequalities: a perturbation approach}, Math. Ann.,  (2016), pp.~1--34.

\bibitem{dolbeault2022parabolic}
{\sc J.~Dolbeault and A.~Zhang}, {\em Parabolic methods for ultraspherical
  interpolation inequalities}.
\newblock Preprint
  \href{https://hal.archives-ouvertes.fr/hal-03573888}{hal-03573888} and
  \href{https://arxiv.org/abs/2202.07041}{arXiv: 2202.07041}, 2022.

\bibitem{Federbush}
{\sc P.~Federbush}, {\em Partially alternate derivation of a result of
  {N}elson}, J. Mathematical Phys., 10 (1969), pp.~50--52.

\bibitem{Felli2003}
{\sc V.~Felli and M.~Schneider}, {\em Perturbation results of critical elliptic
  equations of {C}affarelli-{K}ohn-{N}irenberg type}, J. Differential
  Equations, 191 (2003), pp.~121--142.

\bibitem{MR2198019}
{\sc I.~Gentil, A.~Guillin, and L.~Miclo}, {\em Modified logarithmic {S}obolev
  inequalities and transportation inequalities}, Probab. Theory Related Fields,
  133 (2005), pp.~409--436.

\bibitem{MR2351133}
\leavevmode\vrule height 2pt depth -1.6pt width 23pt, {\em Modified logarithmic
  {S}obolev inequalities in null curvature}, Rev. Mat. Iberoam., 23 (2007),
  pp.~235--258.

\bibitem{MR615628}
{\sc B.~Gidas and J.~Spruck}, {\em Global and local behavior of positive
  solutions of nonlinear elliptic equations}, Comm. Pure Appl. Math., 34
  (1981), pp.~525--598.

\bibitem{Gross75}
{\sc L.~Gross}, {\em Logarithmic {S}obolev inequalities}, Amer. J. Math., 97
  (1975), pp.~1061--1083.

\bibitem{Guionnet-Zegarlinski03}
{\sc A.~Guionnet and B.~Zegarlinski}, {\em Lectures on {Logarithmic} {Sobolev}
  {Inequalities}}, S\'eminaire de probabilit\'es de Strasbourg, 36 (2002),
  pp.~1--134.

\bibitem{MR1731336}
{\sc T.~Horiuchi}, {\em Best constant in weighted {S}obolev inequality with
  weights being powers of distance from the origin}, J. Inequal. Appl., 1
  (1997), pp.~275--292.

\bibitem{Ilyin}
{\sc V.~P. Il'in}, {\em Some integral inequalities and their applications in
  the theory of differentiable functions of several variables}, Mat. Sb.
  (N.S.), 54 (96) (1961), pp.~331--380.

\bibitem{MR3497125}
{\sc A.~J\"{u}ngel}, {\em Entropy methods for diffusive partial differential
  equations}, SpringerBriefs in Mathematics, Springer, [Cham], 2016.

\bibitem{MR1796718}
{\sc R.~Lata{\l}a and K.~Oleszkiewicz}, {\em Between {S}obolev and
  {P}oincar\'e}, in Geometric aspects of functional analysis, vol.~1745 of
  Lecture Notes in Math., Springer, Berlin, 2000, pp.~147--168.

\bibitem{zbMATH04155282}
{\sc P.-L. {Lions}}, {\em {The concentration-compactness principle in the
  calculus of variations. The locally compact case. II}}, {Ann. Inst. Henri
  Poincar\'e, Anal. Non Lin\'eaire}, 1 (1984), pp.~223--283.

\bibitem{MR0343816}
{\sc E.~Nelson}, {\em The free {M}arkoff field}, J. Funct. Anal., 12 (1973),
  pp.~211--227.

\bibitem{MR2352327}
{\sc G.~Royer}, {\em An initiation to logarithmic {S}obolev inequalities},
  vol.~14 of SMF/AMS Texts and Monographs, American Mathematical Society,
  Providence, RI; Soci\'{e}t\'{e} Math\'{e}matique de France, Paris, 2007.
\newblock Translated from the 1999 French original by Donald Babbitt.

\bibitem{MR2001882}
{\sc D.~Smets and M.~Willem}, {\em Partial symmetry and asymptotic behavior for
  some elliptic variational problems}, Calc. Var. Partial Differential
  Equations, 18 (2003), pp.~57--75.

\bibitem{MR0109101}
{\sc A.~J. Stam}, {\em Some inequalities satisfied by the quantities of
  information of {F}isher and {S}hannon}, Information and Control, 2 (1959),
  pp.~101--112.

\bibitem{Wang:1250982}
{\sc F.-Y. Wang}, {\em {Functional inequalities, Markov semigroups and spectral
  theory}}, Mathematics monograph series, Elsevier, Burlington, MA, 2006.

\bibitem{MR479373}
{\sc F.~B. Weissler}, {\em Logarithmic {S}obolev inequalities for the
  heat-diffusion semigroup}, Trans. Amer. Math. Soc., 237 (1978), pp.~255--269.

\end{thebibliography}
\end{document}